\documentclass[11pt]{amsart}
\usepackage{setspace}
\usepackage{xcolor}
\usepackage{makecell}
\usepackage{amssymb,amscd,amsthm,verbatim,amsmath,color,fancyhdr,mathrsfs,amsfonts,amssymb,commath, graphicx,bbm}
\usepackage[bookmarks=false]{hyperref}
\usepackage[english]{babel}
\usepackage{parskip}
\usepackage{enumitem}

\usepackage[letterpaper,top=2cm,bottom=2cm,left=3cm,right=3cm,marginparwidth=1.75cm]{geometry}

\newcommand{\R}{\mathbb{R}}

\newcommand{\of}[1]{\left(#1\right)}

\DeclareMathOperator{\Id}{\mathrm{Id}}

\newcommand{\Was}[1]{\mathbb{W}_{#1}}

\newtheorem{corollary}[]{Corollary}

\DeclareMathOperator*{\argmin}{argmin}

\newcommand{\cP}{\mathcal{P}}

\newcommand{\cU}{\mathcal{U}}
\newcommand{\cO}{\mathcal{O}}

\newcommand{\Ent}{\mathrm{Ent}}

\newcommand{\Tr}{\text{Tr}}

\newcommand{\Tan}[1]{\text{Tan}(#1)}
\newcommand{\Schro}{\text{Schr\"{o}dinger}}

\newcommand{\Exp}[1]{\mathrm{E}_{#1}}

\newcommand{\eps}{\varepsilon}

\newtheorem{lemma}{Lemma}
\newtheorem{remark}{Remark}

\newtheorem{theorem}{Theorem}
\newtheorem{assumption}{Assumption}

\newcommand{\SB}[3]{\text{SB}_{#2}^{#3}(#1)}

\newcommand{\LD}[3]{\text{LD}_{#2}^{#3}(#1)}

\newcommand{\SBstatic}[2]{\pi_{#1,#2}}

\newcommand{\SBdynam}[2]{S_{#1,#2}}
\newcommand{\LDstatic}[2]{\ell_{#1,#2}}

\newcommand{\BMstatic}[2]{\mu_{#1,#2}}

\newcommand{\BPbase}[2]{\mathcal{B}_{#1,#2}}
\newcommand{\LBPbase}[2]{\mathcal{L}_{#1,#2}}

\newcommand{\opt}[3]{S_{#1}^{#2}\of{#3}}

\newcommand{\commentout}[1]{}

\title[Langevin Approximation to Schr\"odinger bridge]{Langevin Diffusion Approximation to Same Marginal Schr\"{o}dinger Bridge}

\author{Medha Agarwal}
\address{Medha Agarwal\\ Department of Statistics \\ University of Washington\\ Seattle WA 98195, USA\\ {Email: medhaaga@uw.edu}}
\author{Zaid Harchaoui}
\address{Zaid Harchaoui\\ Department of Statistics \\ University of Washington\\ Seattle WA 98195, USA\\ {Email: zaid@uw.edu}}
\author{Garrett Mulcahy}
\address{Garrett Mulcahy\\ Department of Mathematics \\ University of Washington\\ Seattle WA 98195, USA\\ {Email: gmulcahy@uw.edu}}
\author{Soumik Pal}
\address{Soumik Pal\\ Department of Mathematics \\ University of Washington\\ Seattle WA 98195, USA\\ {Email: soumik@uw.edu}}

	

\thanks{This research is partially supported by the following grants. All the authors are supported by NSF grant DMS-2134012. Additionally, Pal is supported by NSF grant DMS-2133244; Harchaoui is supported by NSF grant CCF-2019844, DMS-2023166, DMS-2133244. Mulcahy is supported by the National Science Foundation Graduate
Research Fellowship Program under Grant No.\ DGE-2140004. Thanks to PIMS Kantorovich Initiative for facilitating this collaboration supported through a PIMS PRN \#01 and the NSF Infrastructure grant DMS 2133244. The authors are listed in alphabetical order.}

\date{\today}

\begin{document}

\begin{abstract}
We introduce a novel approximation to the same marginal $\Schro$ bridge using the Langevin diffusion. As $\eps \downarrow 0$, it is known that the barycentric projection (also known as the entropic Brenier map) of the $\Schro$ bridge converges to the Brenier map, which is the identity. Our diffusion approximation is leveraged to show that, under suitable assumptions, the difference between the two is $\eps$ times the gradient of the marginal log density (i.e., the score function), in $\mathbf{L}^2$. More generally, we show that the family of Markov operators, indexed by $\eps > 0$, derived from integrating test functions against the conditional density of the static Schr\"{o}dinger bridge at temperature $\eps$, admits a derivative at $\eps=0$ given by the generator of the Langevin semigroup. Hence, these operators satisfy an approximate semigroup property at low temperatures. 
\end{abstract}

\maketitle

\section{Introduction} \label{sec:introduction}
Let $\cP_2(\R^d)$ denote the space of square-integrable Borel probability measures on $\R^d$. Throughout this paper we will let $\rho \in \cP_2(\mathbb{R}^{d})$ refer to both the measure and its Lebesgue density, whenever it exists. $\cP_2(\R^d)$ can be turned into a metric space equipped with the Wasserstein-$2$ metric \cite[Chapter 7]{ambrosio2005gradient}. We will refer to this metric space as simply the Wasserstein space. For any pair of probability measures $(\rho_1, \rho_2)$ in the Wasserstein space, let $\Pi(\rho_1, \rho_2)$ denote the set of couplings, i.e., joint distributions, with marginals $(\rho_1, \rho_2)$. Further, define the relative entropy, also known as Kullback-Leibler (KL) divergence, between $\rho_1$ and $\rho_2$ as
\begin{align}
    H(\rho_1|\rho_2) := \begin{cases}
        \int \log\left(\frac{d\rho_1}{d\rho_2}\right) d\rho_1 &\text{if $\rho_1 << \rho_2$} \\
        +\infty &\text{otherwise}
    \end{cases}
\end{align}

The static Schr\"{o}dinger bridge (\cite{schroLeonard13}) at temperature $\eps$ with both marginals $\rho$ is the solution to the following optimization problem
\begin{align}\label{eq:schrodinger}
    \SBstatic{\rho}{\eps}  &:= \argmin_{\pi \in \Pi(\rho,\rho)} H(\pi|\BMstatic{\rho}{\eps}),
\end{align}
where $\BMstatic{\rho}{\eps} \in \cP(\mathbb{R}^{d} \times \mathbb{R}^{d})$ has density
\begin{align}\label{eq:base-for-SB}
\BMstatic{\rho}{\eps}(x,y) := \rho(x)\frac{1}{(2\pi\eps)^{d/2}}\exp\left(-\frac{1}{2\eps}\|x-y\|^{2}\right).    
\end{align}
If $(X,Y) \sim \SBstatic{\rho}{\eps}{}$, define the \textbf{barycentric projection} \cite{pooladian2022entropic,chewi2022entropic} (also known as the \textbf{entropic Brenier map} \cite{divol2024tightstabilityboundsentropic}) as a function from $\R^d$ to itself given by 
\begin{align}\label{eq:bary-proj}
    \BPbase{\rho}{\eps}(x) &= \Exp{\SBstatic{\rho}{\eps}{}}[Y|X=x].
\end{align}
As $\eps \downarrow 0$ it is known to converge \cite{nutz-weisel-22} in a suitable sense to the Brenier map \cite{pooladian2022entropic}, i.e.\ the solution to the unregularized optimal transport problem. 

The convergence of the entropic Brenier map to the Brenier map, or equivalently the gradient of entropic potentials (also called $\Schro$ potentials, see \eqref{eq:SB-density} for definition) to the gradient of Kantorovich potentials, is a subject of intense focus in the $\Schro$ bridge literature \cite{gigli2021second,chiarini2022gradient,nutz-weisel-22,divol2024tightstabilityboundsentropic}. The results of this paper contribute to this effort in the same-marginal case when the Brenier map is the identity. We show that, under suitable assumptions on $\rho$, the following convergence holds in $\mathbf{L}^2(\rho)$:
\begin{equation}\label{eq:barycenterapprox}
\lim_{\epsilon \downarrow 0}\frac{\BPbase{\rho}{\eps}(x) - x}{\eps} = \frac{1}{2} \nabla \log \rho(x).  
\end{equation}
The quantity $\nabla \log \rho$ is called the \textbf{score function} of $\rho$.  In particular, 
\[
\lim\limits_{\eps \downarrow 0} \eps^{-2}\|\BPbase{\rho}{\eps}-\Id\|_{\mathbf{L}^2(\rho)}^{2} = \frac{1}{4}I(\rho),
\]
where the quantity $I(\rho)$ is the Fisher information of $\rho$.

This result was expected from various recent works using heuristic arguments. For example, the proof of \cite[Theorem 1]{sander_22} notes that \eqref{eq:barycenterapprox} follows from \cite[Theorem 1]{marshall-coifman-19}. But this requires rather restrictive assumptions, such as compactly supported densities and that a certain linear function should be in the span of finitely many Neumann eigenfunctions of the Laplacian associated to the support. Similarly, \cite[Remark 4]{mordant24selfEOT} presents a heuristic argument from Fourier analysis by which one can anticipate \eqref{eq:barycenterapprox}.

One technical novelty in our work is that the proofs are probabilistic, which allows us to vastly generalize \eqref{eq:barycenterapprox}. The foundations of our results can be found in Theorem \ref{thm:ld-sb-general-g}, which shows that the two point joint distribution of a stationary Langevin diffusion is a close approximation of the Schrödinger bridge with equal marginals.
More concretely, let $\rho = e^{-g}$ be some element in $\cP_2(\R^d)$ for which the Langevin diffusion process with stationary distribution $\rho$ exists (weakly) and is unique in law. This diffusion process is governed by the stochastic differential equation 
\begin{equation}\label{eq:langevin_sde}
    dX_t = -\frac{1}{2}\nabla g(X_t)dt + dB_t\,,\quad X_0 \sim \rho,
\end{equation}
where $(B_t, t \geq 0)$ is the standard $d$-dimensional Brownian motion. Therefore, for any $\eps>0$, $\LDstatic{\rho}{\eps} := \text{Law}(X_0,X_{\eps})$ is an element in $\Pi(\rho, \rho)$. Then, Theorem \ref{thm:ld-sb-general-g} shows that, under appropriate conditions on $\rho$, the symmetrized relative entropy, i.e., the Jensen-Shannon divergence, between $\LDstatic{\rho}{\eps}$ and the Schr\"odinger bridge $\SBstatic{\rho}{\eps}$ is $o(\eps^2)$. Although this level of precision is sufficient for our purpose, it is shown via explicit computations in \eqref{eq:gaussiancompute} that when $\rho$ is Gaussian this symmetrized relative entropy is actually $O(\eps^4)$, an order of magnitude tighter.

Using this Langevin approximation, \eqref{eq:barycenterapprox} is made precise in Theorem \ref{thm:score-function-approx}, including an error bound. However, our contributions go further. Let $\xi$ denote a compactly supported smooth real-valued function on $\mathbb{R}^d$. Replace the identity function in \eqref{eq:bary-proj} by $\xi$, i.e., consider $\Exp{\SBstatic{\rho}{\eps}{}}[\xi(Y)|X=x]$. Then, we show in Theorem \ref{thm:sb-generator-approx} that 
\begin{equation}\label{eq:genapproxSB}
\lim_{\eps \downarrow 0} \frac{\Exp{\SBstatic{\rho}{\eps}{}}[\xi(Y)|X=x] - \xi(x)}{\eps}= L\xi(x),
\end{equation}
where $L$ is the generator of the Langevin diffusion acting on smooth functions. The limit \eqref{eq:genapproxSB} tells us that the family of integral operators derived from the same-marginal Schr\"{o}dinger bridge for different $\eps$ forms an approximate semigroup as $\eps \rightarrow 0+$ with a generator given by the Langevin diffusion.


Another consequence of our Langevin approximation to the Schr\"odinger bridge is to show that ``tangential approximations'' to absolutely continuous curves in the Wasserstein spaces may be done using entropic Brenier potentials, which has applications to computational optimal transport. 
Absolutely continuous curves (see \cite[Definition 1.1.1]{ambrosio2005gradient}) $\left(\rho_t,\; t\ge 0 \right)$ on the Wasserstein space are characterized as weak solutions \cite[eqn.\ (8.1.4)]{ambrosio2005gradient} of continuity equations
\begin{align}\label{eq:continuity-equation}
   \frac{\partial}{\partial t} \rho_t + \nabla \cdot \of{v_t \rho_t} = 0\,.
\end{align}
Here $v_t: \R^d \to \R^d$ is a time-dependent Borel velocity field of the continuity equation whose properties are given by \cite[Theorem 8.3.1]{ambrosio2005gradient}. For any absolutely continuous curve, if $v_t$ belongs to the tangent space at $\rho_t$ (\cite[Definition 8.4.1]{ambrosio2005gradient}), \cite[Proposition 8.4.6]{ambrosio2005gradient} states the following approximation result for Lebesgue-a.e. $t \geq 0$
\begin{align}\label{eq:ags8-4-6}
    \lim\limits_{\eps \to 0} \frac{1}{\eps}\Was{2}(\rho_{t+\eps},(\Id+\eps v_t)_{\#}\rho_t) = 0. 
\end{align}
For our purpose, assume that the above limit holds at time $t = 0$. This can be thought of as a Wasserstein counterpart of the Euclidean tangential approximation and is often used in discretization schemes (such as the explicit Euler scheme) of Wasserstein gradient flows. 

An example of this is the (probabilists') heat equation starting from $\rho$, 
\[
\frac{\partial}{\partial t} \rho_t - \frac{1}{2}\Delta \rho_t = 0, \quad \text{with tangent velocity}\quad v_t = -\frac{1}{2}\nabla \log \rho_t.
\]
The tangential approximation in \eqref{eq:ags8-4-6} at $t=0$ implies that 
\begin{align}\label{eq:tangentheat}
    \lim\limits_{\eps \to 0} \frac{1}{\eps}\Was{2}\left(\rho_\eps,\left(\Id-\frac{\eps}{2} \nabla \log \rho \right)_{\#}\rho\right) = 0. 
\end{align}
However, due to \eqref{eq:barycenterapprox}, one expects $2x- \BPbase{\rho}{\eps}(x) \approx x - \frac{\eps}{2} \nabla \log \rho(x)$, and hence
\begin{align}\label{eq:eq:tangentexpeuler}
    \lim\limits_{\eps \to 0} \frac{1}{\eps}\Was{2}\left(\left(\Id-\frac{\eps}{2} \nabla \log \rho \right)_{\#}\rho,\left(2\Id- \BPbase{\rho}{\eps}\right)_{\#}\rho\right) = 0. 
\end{align}
Hence, one may be able to replace the tangential approximation \eqref{eq:tangentheat} by a Schr\"odinger bridge approximation in the following sense. 
\begin{align*}
    \lim\limits_{\eps \to 0} \frac{1}{\eps}\Was{2}\left(\rho_\eps,\left(2\Id- \BPbase{\rho}{\eps}\right)_{\#}\rho\right) = 0. 
\end{align*}
The advantage of this approach is that while the score function is notoriously difficult to estimate from data, an approximate barycentric projection is readily available due to the Sinkhorn algorithm. Of course, this method is only useful if it can be replicated for a wide class of tangent velocities. Our final contribution in this paper is to show that this is indeed true.

For the moment let $v \in \Tan{\rho}$ be such that $v = \nabla u$ for some smooth $u \in C^{\infty}(\mathbb{R}^{d})$. When this holds, we say that $v$ is of \emph{gradient type}. By definition, such velocity fields are dense in $\Tan{\rho}$. Suppose that there is some $\theta \in \mathbb{R} \setminus \{0\}$ such that $\int e^{2\theta u(x)}dx < \infty$, define the \textbf{surrogate measure} $\sigma \in \cP_{2}^{ac}(\mathbb{R}^{d})$ to be one with the density
\begin{align}\label{eq:surrogate}
    \sigma(x) &:= \exp\left(2\theta u(x)- \Lambda(\theta)\right), \quad \text{where}\quad \Lambda(\theta):=\log \int e^{2\theta u(y)}dy.
\end{align}

Define \eqref{eq:ags8-4-6} the \textbf{explicit Euler update} as the pushforward
\begin{align}\label{eq:exp-euler}
    \opt{\eps}{1}{\rho} := (\Id + \eps v)_{\#}\rho.
\end{align}

 If we consider the Schr\"odinger bridge at temperature $\eps$ with both marginals to be $\sigma$, then, by \eqref{eq:barycenterapprox}, the following approximation holds in $\mathbf{L}^2(\sigma)$:
\begin{align}\label{eq:bp-base-sigma}
    \BPbase{\sigma}{\eps}(x) \approx x+ \frac{\eps}{2}\nabla \log \sigma(x) &= x + \theta \eps\nabla u(x). \\
    \text{Rearranging terms,}\quad x + \frac{\BPbase{\sigma}{\eps}(x) - x}{\theta} &\approx x + \epsilon \nabla u(x). 
\end{align}
Note that $\theta=-1$ and $v=-\frac{1}{2}\nabla \log \rho$ gives us \eqref{eq:eq:tangentexpeuler}. In this case, $\sigma=\rho$.

But by definition, velocity fields of gradient type are $\mathbf{L}^2(\rho)$-\emph{dense} in $\Tan{\rho}$. Hence, by constructing a sequence of surrogate measures and their corresponding entropic Brenier maps, we can approximate arbitrary elements of $\Tan{\rho}$. This result is precisely the statement of Theorem \ref{thm:one_step_convergence}.

\textbf{Discussion on Same Marginal $\Schro$ Bridges}
The same marginal $\Schro$ bridge has seen a recent surge of attention in works such as \cite{sander_22,lavenant2024sinkhorndiv,mordant24selfEOT}. At first glance, this attention may appear unwarranted as the unregularized same-marginal optimal transport problem is completely trivial. However, the \emph{deviation} of the same marginal $\Schro$ bridge at small temperature from the unregularized solution (i.e.\ $\Id$) has surprising connections to machine learning \cite{sander_22}, statistics \cite{mordant24selfEOT}, and Wasserstein gradient flows \cite{AHMP}. 

These connections are rooted in the appearance of the score function as the first order deviation in $\eps$ of the entropic Brenier map from the identity map. \cite{sander_22} uses this intuition to derive a deep connection between a modified self-attention mechanism of the Transformer neural architecture and the heat flow. Similarly, \cite{mordant24selfEOT} develops statistical results based on using the entropic Brenier map as a nonparametric (albeit biased) estimator of the score function, which is of great interest in machine learning (see diffusion models, \cite{song2020score}). While perhaps less connected to the aforementioned works, \cite{lavenant2024sinkhorndiv} uses the Sinkhorn divergence to develop a novel geometric structure on the space of probability measures, which due to the debiasing term in the Sinkhorn divergence necessitates a closer analysis of the same-marginal $\Schro$ bridge. Overall, the results presented in our paper closely connect with those presented in \cite{sander_22}. Theorems~\ref{thm:ld-sb-general-g}-\ref{thm:one_step_convergence} can be understood as probabilistic approaches to~\cite[Theorem 1]{sander_22}, as well as a finer treatment to key steps of the mathematical analysis of the convergence to the heat flow. The first statement of \cite[Theorem 1]{sander_22} concerns the convergence of the rescaled gradients of entropic potentials for the $\Schro$ bridge with the same marginals. Consider the so-called entropic (or, Schr\"odinger) potentials $(f_{\eps}, \eps > 0)$ as defined later on in \eqref{eq:SB-density}. In this language, \cite[Theorem 1]{sander_22} states that $\eps^{-1} \nabla f_{\eps} \to -\nabla \log \rho$ in $\textbf{L}^{2}(\rho)$, as $\eps \rightarrow 0+$, when $\rho$ is compactly supported and satisfies the assumptions of \cite[Theorem 1]{marshall-coifman-19}. 
We prove this result in the case of fully supported $\rho$ and with less taxing technical assumptions in Theorem \ref{thm:score-function-approx}. 

In follow up work, \cite{MP25} generalizes Theorems \ref{thm:ld-sb-general-g}-\ref{thm:sb-generator-approx} to the manifold setting.

\textbf{Outline of the paper.}
Section \ref{sec:preliminaries} presents a self-contained overview of the relevant notions of $\Schro$ bridges and the Langevin diffusion to present and prove the results in this paper. Section \ref{sec:sb-ld} is devoted to proving the Langevin approximation to the same-marginal Schr\"odinger bridge, Theorem \ref{thm:ld-sb-general-g}. In Section \ref{sec:one-step}, this approximation is then leveraged to show that \eqref{eq:barycenterapprox} holds in the sense made precise by Theorem \ref{thm:score-function-approx}. Similarly, Theorem \ref{thm:sb-generator-approx} uses this approximation to establish the generator convergence in \eqref{eq:genapproxSB}. Finally, Theorem \ref{thm:one_step_convergence} demonstrates that a well-chosen sequence of barycentric projections can approximate arbitrary elements of the tangent space at a given measure satisfying smoothness and integrability assumptions. 

Some of the technical challenges we overcome in this paper involve proving results about Schr\"odinger bridges at low temperatures when the marginals are not compactly supported. The latter is a common assumption in the literature, see for example \cite{pal2019difference,pooladian2022entropic,sander_22}. However, the existence of Langevin diffusion requires an unbounded support.  Additionally, in Corollary \ref{cor:CTextension} we extend the expansion of entropic OT cost function from \cite[Theorem 1.6]{conforti21deriv} for the case of same marginals in \eqref{eq:CTextension}. 

\section{Preliminaries}\label{sec:preliminaries}
In this section we fix a measure $\rho \in \cP_{2}(\mathbb{R}^{d})$ with density given by $\rho = e^{-g}$ for some function $g: \mathbb{R}^{d} \to \mathbb{R} \cup \{+\infty\}$. Here, we follow the standard convention in optimal transport of using the same notation to refer to an absolutely continuous measure and its density. We will state assumptions on $g$ later. Let $\cP_{2}^{ac}(\mathbb{R}^{d})$ denote the collection of probability measures on $\mathbb{R}^{d}$ with density and finite second moments.

In this section, we collate standard results about the Langevin diffusion and $\Schro$ bridge and introduce the notions from stochastic calculus of which we will make use.

Let $C^{d}[0,\infty)$ denote the set of continuous paths from $[0,\infty)$ to $\mathbb{R}^{d}$, equipped with the locally uniform metric. Denote the coordinate process by $(\omega_t, t \geq 0)$. Endow this space with a filtration satisfying the \textit{usual conditions} \cite[Chapter 1, Definition 2.25]{karatshreve91}. For $x \in \mathbb{R}^{d}$, let $W_{x}$ denote the law on $C^{d}[0,\infty)$ of the standard $d$-dimensional Brownian motion started from $x$. All stochastic integrals are It\^{o} integrals.

Let $g \in C^2(\mathbb{R}^{d})$ and assume that $(Y_t, t \geq 0)$, the stationary Langevin diffusion with stationary measure equal to $\rho = e^{-g}$, exists. That is, its law on $C^{d}[0,\infty)$ is a weak solution to (\ref{eq:langevin_sde}) with initial condition $Y_0 \sim \rho$. Let $Q$ denote the law of this process on $C^{d}[0,\infty)$, and let $Q_x$ denote the law of the Langevin diffusion started from $x$. As in the Introduction, for $\eps > 0$, set $\LDstatic{\rho}{\eps} = \text{Law}(Y_0,Y_\eps)$. Under suitable assumptions on $g$ (see Assumption \ref{assumption:LD_SB_relative_entropy} below), $(Y_t, t \geq 0)$ exists with a.s.\ infinite explosion time \cite[Theorem 2.2.19]{royer-lsi}, and by \cite[Lemma 2.2.21]{royer-lsi} for each time $t > 0$ the laws of $Q_x$ and $P_x$ restricted to $C^{d}[0,t]$ are mutually absolutely continuous with Radon-Nikodym derivative 
\begin{align}\label{eq:rn-LD-BM}
    \left.\frac{dQ_x}{dW_x}\right|_{C^{d}[0,t]}(\omega) &= \sqrt{\frac{\rho(\omega_t)}{\rho(x)}}\exp\left(-\int_{0}^{t} \left(\frac{1}{8}\|\nabla g(\omega_s)\|^2-\frac{1}{4}\Delta g(\omega_s)\right) ds\right). 
\end{align}
The expression inside the integral of the $\exp(\cdot)$ is a distinguished quantity in the Langevin diffusion literature (see \cite{LKreener}). Denote it by 
\begin{align}\label{eq:harmoniccharacteriestic}
    \cU(x) &:= \frac{1}{8}\|\nabla g(x)\|^2 - \frac{1}{4}\Delta g(x). 
\end{align}
This function is called the harmonic characteristic in \cite[eqn.\ (2)]{vonrenesse-conf18}, which provides a heuristic physical interpretation for $\cU$ and $\nabla \cU$ (the latter of which appears in Theorem \ref{thm:ld-sb-general-g}) as the ``mean acceleration'' of the Langevin bridge. Its interpretation as an acceleration is also explored in \cite{conforti2019second}. 

Let $(p_t(\cdot,\cdot), t > 0)$ and $(q_t(\cdot,\cdot), t > 0)$ denote the transition kernels for the standard $d$-dimensional Brownian motion and the Langevin diffusion, respectively. By \eqref{eq:rn-LD-BM},
\begin{align*}
    \frac{q_{t}(x,y)}{p_{t}(x,y)} &= \sqrt{\frac{\rho(y)}{\rho(x)}}\Exp{W_{x}}\left[\exp\left(-\int_{0}^{t} \cU(\omega_s) ds\right)\middle| \; \omega_t = y\right].
\end{align*}
Notice that $\Exp{W_x}[\cdot|\omega_t = y]$ is the expectation with respect to the law of the Brownian bridge from $x$ to $y$ over $[0,t]$. Define the following function
\begin{align}\label{eq:bridge-exp}
    c(x,y,\eps) := -\log\left(\Exp{W_{x}}\left[\exp\left(-\int_{0}^{\eps}\cU(\omega_s) ds\right)\middle| \; \omega_\eps = y\right]\right). 
\end{align}
Recall that $\LDstatic{\rho}{\eps}$ is the joint density of $(Y_0,Y_{\eps})$ under Langevin; it is now equal to
\begin{align}\label{eq:LD-density}
    \LDstatic{\rho}{\eps}(x,y) &= \rho(x)q_{\eps}(x,y) = \sqrt{\rho(x)\rho(y)}p_{\eps}(x,y)\exp(-c(x,y,\eps)) \\
    &= \sqrt{\rho(x)\rho(y)}\frac{1}{(2\pi\eps)^{d/2}}\exp\left(-\frac{1}{2\eps}\|x-y\|^2 -c(x,y,\eps)\right).
\end{align}

Let $(G_t, t \geq 0)$ denote the corresponding Langevin semigroup, $L$ its infinitesimal generator, and $\mathcal{D}(L)$ the domain of $L$ (see \cite[Section 1.14]{bgl-markov} for definitions). That is, for $t \geq 0$ and $f: \mathbb{R}^{d} \to \mathbb{R}$ bounded and measurable,
\begin{align*}
    G_{t}f(x) := \Exp{Q}[f(\omega_t)|\omega_0 = x].
\end{align*}
Let $C^{2}(\mathbb{R}^{d})$ denote the set of twice continuously differentiable functions, by \cite[Chapter VIII, Proposition 3.4]{revuz2004continuous} for $f \in C^{2}(\mathbb{R}^{d})$ it holds that
\begin{align}\label{eq:LD-generator}
    Lf &= \frac{1}{2}\Delta f + \left(-\frac{1}{2}\nabla g\right) \cdot \nabla f. 
\end{align}

We now define the static $\Schro$ bridge with marginals equal to $\rho$, denoted by $\SBstatic{\rho}{\eps}$. For an excellent and detailed account we refer readers to \cite{schroLeonard13} and references therein; we develop only the notions that we use within our proof. Under mild regularity assumptions on $\rho$ (for instance, finite entropy and a moment condition), it is known that for each $\eps > 0$, $\SBstatic{\rho}{\eps}$ admits an $(f,g)$-decomposition \cite[Theorem 2.8]{schroLeonard13}. As we consider the equal marginal $\Schro$ bridge, we can in fact insist that $f = g$. Thus, there exists $f_{\eps}: \mathbb{R}^{d} \to \mathbb{R}$ such that $\SBstatic{\rho}{\eps}$ admits a Lebesgue density of the form
\begin{align}\label{eq:SB-density}
    \SBstatic{\rho}{\eps}(x,y) &= \frac{1}{(2\pi\eps)^{d/2}}\exp\left(\frac{1}{\eps}\left(f_{\eps}(x)+f_{\eps}(y)-\frac{1}{2}\|x-y\|^2\right)\right)\rho(x)\rho(y). 
\end{align}
We refer to the $(f_{\eps}, \eps > 0)$ as entropic potentials (in the literature they are also referred to as $\Schro$ potentials). The entropic potentials are related to the barycentric projection by the following identity \cite[eqn.\ (8)]{chewi2022entropic}:
\begin{align}\label{eq:bp-potential-identity}
    \Exp{\SBstatic{\rho}{\eps}}[Y|X=x] = \nabla \varphi_{\eps}(x) = (\Id - \nabla f_{\eps})(x).
\end{align}
Note that, in general, entropic potentials are unique up to a choice of constant; however, our convention of symmetry forces that each $f_{\eps}$ is indeed unique.

As outlined in \cite[Proposition 2.3]{schroLeonard13}, there is a corresponding dynamic $\Schro$ bridge, obtained from $\SBstatic{\rho}{\eps}$ (the ``static'' $\Schro$ bridge) by interpolating with suitably rescaled Brownian bridges. More precisely, let $R_{\eps}$ denote the law of $(\sqrt{\eps}B_t, t \in [0,1])$ on $C^{d}[0,1]$, where $(B_t, t \geq 0)$ is the reversible $d$-dimensional Brownian motion. Then, the law of the $\eps$-dynamic $\Schro$ bridge on $C^{d}[0,1]$ is given by $\SBdynam{\rho}{\eps}$ where
\begin{align}\label{eq:dynSB-defn}
    \frac{d\SBdynam{\rho}{\eps}{}}{dR_{\eps}}(\omega) = a^{\eps}(\omega_0)a^{\eps}(\omega_1),
\end{align}
where $a^{\eps}(x) := \exp\left(\frac{1}{\eps}f_{\eps}(x)+\log \rho(x)\right)$. Letting $(P_t, t \geq 0)$ denote the (standard) heat semigroup, we define
\begin{align}\label{eq:SB-psi}
    \psi^{\eps}(t,x) &:= \log E_{R_{\eps}}\left[a^{\eps}(\omega_1)|\omega_t = x\right] = \log E_{R_{\eps}}\left[a^{\eps}(\omega_{1-t})|\omega_0 = x\right] = \log P_{\eps(1-t)}a^{\eps}(x),
\end{align}
where the second identity follows from the Markov property of $R_{\eps}$. 
 
Fix $\eps > 0$, set $\rho_{s}^{\eps} := (\pi_{s})_{\#}\SBdynam{\rho}{\eps}$ and call $(\rho_{s}^{\eps}, s \in [0,1])$ the $\eps$-\textbf{entropic interpolation} from $\rho$ to itself, explicitly
\begin{align}\label{eq:entropic-inter}
    \rho_{t}^{\eps} = P_{\eps t}a^{\eps}P_{\eps(1-t)}a^{\eps} = \exp(\psi^{\eps}(t,\cdot)+\psi^{\eps}(1-t,\cdot)).
\end{align}
By \cite[Proposition 2.3]{schroLeonard13}, we construct a random variable with law $\rho_{t}^{\eps}$ as follows. Let $(X_{\eps},Y_{\eps}) \sim \SBstatic{\rho}{\eps}$ and $Z \sim N(0,\Id)$ be independent to the pair, then
\begin{align}\label{eq:rv-entropic-inter}
    X_{t}^{\eps} = (1-t)X_{\eps}+t Y_{\eps}+\sqrt{\eps t(1-t)}Z \sim \rho_{t}^{\eps}.
\end{align}
Thanks to an analogue of the celebrated Benamou-Brenier formula for entropic cost (\cite[Theorem 5.1]{gentil2017analogy}), the entropic interpolation is in fact an absolutely continuous curve in the Wasserstein space. From \cite[Equation (2.14)]{conforti21deriv} the entropic interpolation satisfies the continuity equation \eqref{eq:continuity-equation} with $(\rho_t^{\eps},v_{t}^{\eps})_{t \in [0,1]}$ given by
\begin{align}\label{eq:cont-eq-entropic-inter}
    v_{t}^{\eps} = \frac{\eps}{2}\left(\nabla \psi^{\eps}(t,\cdot)-\nabla \psi^{\eps}(1-t,\cdot)\right).
\end{align}
Moreover, under mild regularity assumptions on $\rho$, we obtain the following entropic variant of Benamou-Brenier formulation \cite[Corollary 5.8]{gentil2017analogy} (with appropriate rescaling)
\begin{align}\label{eq:symm-ent-cost-BB}
    H(\SBstatic{\rho}{\eps}{}|\BMstatic{\rho}{\eps}{}) &= \inf\limits_{(\rho_t,v_t)} \left(\frac{1}{2\eps}\int_{0}^{1} \|v_t\|_{L^{2}(\rho_t)}^2 dt+\frac{\eps}{8}\int_{0}^{1} I(\rho_t)dt\right),
\end{align}
where $I(\cdot)$ is the Fisher information, defined for $\sigma \in \cP_{2}^{ac}(\mathbb{R}^{d})$ by $I(\sigma) = \int \|\nabla \log \sigma\|^{2} d\sigma$. The infimimum in \eqref{eq:symm-ent-cost-BB} is taken over all $(\rho_t,v_t)_{t \in [0,1]}$ satisfying $\rho_0 = \rho_1 = \rho$ and \eqref{eq:continuity-equation}. The optimal selection is given in \eqref{eq:cont-eq-entropic-inter}. That is, the optimal selection is the entropic interpolation. For an AC curve $(\rho_t, t \in [0,1])$, we will call the quantity $\int_{0}^{1} I(\rho_{t}) dt$ the integrated Fisher information. 

Lastly, for the same marginal dynamic $\Schro$ bridge, we make the following observation about the continuity at $\eps = 0$ of the integrated Fisher information along the entropic interpolation. 
\begin{lemma}\label{lem:integrated-fisher-continuity}
Let $\rho \in \cP_{2}(\mathbb{R}^{d})$ satisfy $-\infty <\Ent(\rho) < +\infty$ and $I(\rho) <+\infty$, then
    $\lim\limits_{\eps \downarrow 0} \int_{0}^{1} I(\rho_{t}^{\eps}) dt = I(\rho)$. 
\end{lemma}
\begin{proof}
As $-\infty < \Ent(\rho) < +\infty$, the entropic interpolation $(\rho_{s}^{\eps}, s \in [0,1])$ exists for all $\eps > 0$. By \cite[Theorem 3.7(3)]{leo-sb-to-kp12}, for each $s \in [0,1]$, $\rho_s^\eps$ converges weakly to $\rho$ as $\eps \downarrow 0$. By the lower semicontinuity of Fisher information with respect to weak convergence \cite[Proposition 14.2]{bobkov-fisher-22} and Fatou's Lemma,
\begin{align*}
    I(\rho) &= \int_{0}^{1} I(\rho) dt \leq \int_0^1 \liminf\limits_{\eps \downarrow 0} I(\rho_{t}^{\eps}) dt \leq \liminf\limits_{\eps \downarrow 0} \int_0^1 I(\rho_{t}^{\eps}) dt.
\end{align*}
Lastly, by \eqref{eq:symm-ent-cost-BB} and the submoptimality of the constant curve $\rho_t = \rho$, $v_t = 0$ for $t \in [0,1]$
\begin{align*}
    \frac{1}{2\eps}\int_{0}^{1} \|v_{t}^{\eps}\|_{L^{2}(\rho_t^{\eps})}^2 dt + \frac{\eps}{8}\int_{0}^{1}I(\rho_{t}^{\eps})dt \leq \frac{\eps}{8}I(\rho) \Rightarrow \int_{0}^{1} I(\rho_{t}^{\eps})dt \leq I(\rho).
\end{align*}
Altogether, it follows
\begin{align*}
    0 \leq \limsup\limits_{\eps \downarrow 0} \left( I(\rho)- \int_0^1 I(\rho_{t}^{\eps}) dt\right) &= I(\rho) - \liminf\limits_{\eps \downarrow 0} \int_0^1 I(\rho_{t}^{\eps}) dt \leq I(\rho) - I(\rho) = 0.
\end{align*}
\end{proof}

\section{Approximation of Schr\"{o}dinger Bridge by Langevin Diffusion}\label{sec:sb-ld}
Our main result of the section is the relative entropy bound between $\LDstatic{\rho}{\eps}$ and $\SBstatic{\rho}{\eps}$ given in Theorem \ref{thm:ld-sb-general-g}.

\begin{assumption}\label{assumption:LD_SB_relative_entropy}
Let $\rho = e^{-g}$ for $g: \mathbb{R}^{d} \to \mathbb{R}$, and recall that $\cU = \frac{1}{8}\|\nabla g\|^2-\frac{1}{4}\Delta g$. Assume that
\begin{itemize}
    \item[(1)] $\rho$ is subexponential (\cite[Definition 2.7.5]{vershynin-hdp}), $g \in C^{3}(\mathbb{R}^{d})$, $\cU$ is bounded below, $g \to \infty$ as $\|x\| \to \infty$.
    \item[(2)] There exists $C > 0$ and $n \geq 1$ such that $\|\nabla \cU(x)\|^2 \leq C(1+\|x\|^n)$.
    
    \item[(3)] $-\infty <\Ent(\rho) < +\infty$ and $I(\rho) < +\infty$.
\end{itemize}
\end{assumption}
Note that Assumption \ref{assumption:LD_SB_relative_entropy} (3) ensures the existence of $\SBstatic{\rho}{\eps}$ and that \eqref{eq:symm-ent-cost-BB} holds \cite[(Exi),(Reg1),(Reg2)]{gentil2017analogy}.

These assumptions are, for instance, satisfied by the multivariate Gaussian $N(\mu,\Sigma)$ with $\mu \in \mathbb{R}^{d}$, $\Sigma \in \mathbb{R}^{d \times d}$ positive definite. This provides an essential class of examples as explicit computations of the $\Schro$ bridge between Gaussians are known \cite{janati2020}. In this instance, $g(x) = \frac{1}{2}(x-\mu)^{T}\Sigma^{-1}(x-\mu)+\frac{1}{2}\log((2\pi)^d \det \Sigma)$ and thus
\begin{align*}
    \cU(x) = \frac{1}{8}(x-\mu)^{T}\Sigma^{-2}(x-\mu)-\frac{1}{4}\Tr(\Sigma^{-1}) &\text{ and } \nabla \cU(x) = \frac{1}{4}\Sigma^{-2}(x-\mu). 
\end{align*}
Importantly, for $X \sim N(\mu,\Sigma)$ this establishes that $\nabla \cU(X) \sim N(0,\frac{1}{16}\Sigma^{-3})$ and thus $\Exp{}\|\nabla \cU(X)\|^{2} = \frac{1}{16} \Tr(\Sigma^{-3})$. More generally, when $g$ is polynomial such that $\int_{\mathbb{R}^{d}} e^{-g} < +\infty$ these assumptions are also satisfied.

Additionally, we remark that the class of functions $g$ satisfying Assumption \ref{assumption:LD_SB_relative_entropy} is stable under the addition of functions $h \in C^{3}(\mathbb{R}^{d})$ such that $h$ and all of its derivatives up to and including order $3$ are bounded (e.g.\ $h \in C_{c}^{3}(\mathbb{R}^{d})$). It is also stable under the addition of linear and quadratic functions, so long as the resulting function has a finite integral when exponentiated.

We begin with a preparatory lemma establishing some important integrability properties we make use of in the proof of Theorem \ref{thm:ld-sb-general-g}.
\begin{lemma}\label{lem:integ-entropic-inter}
Let $\rho = e^{-g} \in \cP_{2}(\mathbb{R}^{d})$ satisfying Assumption \ref{assumption:LD_SB_relative_entropy}.
\begin{itemize}
    \item[(1)] For each $\eps > 0$, the densities along the entropic interpolation $(\rho_{s}^{\eps}, s \in [0,1])$ are uniformly subexponential, meaning that for each $T > 0$ there is a constant $K > 0$ such that for $X_{s}^{\eps} \sim \rho_{s}^{\eps}$ we have $\sup\limits_{\eps \in (0,T)}\sup\limits_{s \in [0,1]} \|X_{s}^{\eps}\|_{\psi_{1}} \leq K$, where $\|.\|_{\psi_{1}}$ is the subexponential norm defined in \cite[Definition 2.7.5]{vershynin-hdp}. 
    \item[(2)] The following integrability conditions hold:
    \begin{itemize}
        \item[(a)] The function $\cU = \frac{1}{8}\|\nabla g\|^{2}-\frac{1}{4}\Delta g$ is in $L^{2}(\rho_{s}^{\eps})$ for all $s \in [0,1]$.
        \item[(b)] The function from $[0,1] \times C^{d}[0,1]$ to $\mathbb{R}$ given by $(t,\omega) \mapsto \cU(\omega_t)$ is in $L^{2}(dt \otimes d\SBdynam{\rho}{\eps})$. Similarly, $(t,\omega) \mapsto \nabla \cU(\omega_t)$ is in $L^{2}(dt \otimes d\SBdynam{\rho}{\eps})$.
        \item[(c)] For any $T > 0$, $\sup\limits_{\eps \in (0,T)} \int_{0}^{1} \Exp{\SBdynam{\rho}{\eps}}[\|\nabla \cU(\omega_s)\|^2]ds < +\infty$.
    \end{itemize}
\end{itemize}
\end{lemma}
\begin{proof}
Fix $\eps > 0$ and $s \in [0,1]$. Let $(X,Y)$ be random variables with joint law $\SBstatic{\rho}{\eps}$. Let $Z \sim N(0,\Id)$ be independent of $(X,Y)$. Then $X_{s}^{\eps}$ as given in \eqref{eq:rv-entropic-inter} has law $\rho_{s}^{\eps}$. 
As $\|\cdot\|_{\psi_{1}}$ is a norm, observe that for any $T > 0$
\begin{align*}
    \sup\limits_{\eps \in (0,T], s \in [0,1]}\|X_{s}^{\eps}\|_{\psi_{1}} &\leq \sup\limits_{\eps \in (0,T], s \in [0,1]} \left((1-s)\|X\|_{\psi_{1}}+s\|Y\|_{\psi_{1}}+\sqrt{\eps s(1-s)}\|Z\|_{\psi_{1}}\right) < +\infty 
\end{align*}
establishing the uniform upper bound on the subexponential norm. 

It then follows that
\begin{align*}
    \sup\limits_{\eps \in (0,T], s \in [0,1]}\Exp{}[\|\nabla \cU(X^{\eps}_{s})\|^2] &\leq \sup\limits_{\eps \in (0,T], s \in [0,1]}C(1+\Exp{}[\|X_{s}^{\eps}\|^n]) <+\infty. 
\end{align*} 
This observation also establishes that $(t,\omega) \mapsto \nabla \cU(\omega_t)$ is in $L^{2}(dt \otimes d\SBdynam{\rho}{\eps})$. 
By assumption on $\nabla \cU$, $\cU$ is also of uniform polynomial growth, and thus the same argument establishes the remaining claims. 
\end{proof}

\begin{theorem}\label{thm:ld-sb-general-g}
    For $\rho \in \cP_2(\R^d)$ satisfying Assumption \ref{assumption:LD_SB_relative_entropy}(1,3), for all $\eps > 0$
    \begin{align}\label{eq:rel-ent-ld-sb}
        H(\LDstatic{\rho}{\eps} | \SBstatic{\rho}{\eps})+H(\SBstatic{\rho}{\eps} | \LDstatic{\rho}{\eps}) \leq \frac{1}{2}\eps^2 \left(I(\rho) - \int_{0}^{1}I(\rho_t^\eps) dt\right)^{1/2}\left(\int_{0}^{1} \Exp{\rho_t^\eps}\|\nabla \cU\|^2 dt\right)^{1/2}.
    \end{align}
    Thus, under Assumption \ref{assumption:LD_SB_relative_entropy}(2), 
    $
    \lim\limits_{\eps \downarrow 0} \eps^{-2}\left(H(\LDstatic{\rho}{\eps} | \SBstatic{\rho}{\eps})+H(\SBstatic{\rho}{\eps} | \LDstatic{\rho}{\eps})\right) = 0.
    $
\end{theorem}
Before proving the theorem we pause to comment on the sharpness of the inequality presented in \eqref{eq:rel-ent-ld-sb}. As the $\Schro$ bridge is explicitly known for Gaussian marginals \cite{janati2020}, consider the case $\rho = N(0,1)$. Define two matrices
\begin{align*}
    \Sigma_1 =  \begin{pmatrix}
        1 & e^{-\eps/2} \\ e^{-\eps/2} & 1
    \end{pmatrix}, \Sigma_2 = \begin{pmatrix}
        1 & \frac{1}{2}(\sqrt{\eps^2+4}-\eps) \\ \frac{1}{2}(\sqrt{\eps^2+4}-\eps) & 1 
    \end{pmatrix}.
\end{align*}
It is well known that $\LDstatic{\rho}{\eps} \sim N(0,\Sigma_1)$ is the joint density of the stationary Ornstein-Uhlenbeck process at times $0$ and $\eps$. By \cite{janati2020}, it is known that $\SBstatic{\rho}{\eps} \sim N(0,\Sigma_2)$. One then computes
\begin{align}\label{eq:gaussiancompute}
    H(\LDstatic{\rho}{\eps}|\SBstatic{\rho}{\eps}) + H(\SBstatic{\rho}{\eps}|\LDstatic{\rho}{\eps}) = \frac{1}{2}\Tr(\Sigma_{1}^{-1}\Sigma_{2})+\frac{1}{2}\Tr(\Sigma_{2}^{-1}\Sigma_{1}) - 2= \frac{1}{1152}\eps^{4}+\cO(\eps^{5}). 
\end{align}
On the other hand, the entropic interpolation between Gaussians is computed in \cite{gentil2017analogy}, and $I(\rho_{t}^{\eps})$ is minimized at $t = 1/2$. From these computations, it can be shown that
\begin{align*}
    I(\rho) - \int_{0}^{1}I(\rho_t^\eps) dt \leq I(\rho) - I(\rho_{1/2}^{\eps}) = \cO(\eps^2). 
\end{align*}
Altogether, the RHS of \eqref{eq:rel-ent-ld-sb} is then $\cO(\eps^3)$. That is, the transition densities for the OU process provide a tighter approximation (an order of magnitude) for same marginal $\Schro$ bridge than given by Theorem \ref{thm:ld-sb-general-g}.

Lastly, we note that it would be desirable to have quantitative bounds on the right hand side of \eqref{eq:rel-ent-ld-sb}. This is a delicate task as quantitative results concerning the convergence of Fisher information do not appear to be established outside of specific contexts that do not include the one considered in this paper, see for instance \cite[Theorem 1.5]{conforti2019second}. 

\begin{proof}[Proof of Theorem \ref{thm:ld-sb-general-g}]
We argue in the following manner. First, using the densities for $\SBstatic{\rho}{\eps}$ and $\LDstatic{\rho}{\eps}$ given in (\ref{eq:LD-density}) and (\ref{eq:SB-density}), respectively, we compute their likelihood ratio and obtain exact expressions for the two relative entropy terms appearing in the statement of the Theorem. The resulting expression is the difference in expectation of the expression $c(x,y,\eps)$ defined in (\ref{eq:bridge-exp}) with respect to the Schr\"{o}dinger bridge and Langevin diffusion. We then extract out the leading order terms from both expectations and show that they cancel. To obtain the stated rate of decay, we then analyze the remaining term with an analytic argument from the continuity equation. 

\noindent\textbf{Step 1.} The following holds    
    \begin{align}\label{eq:sym-ent-identity}
        H(\LDstatic{\rho}{\eps}{} | \SBstatic{\rho}{\eps})+H(\SBstatic{\rho}{\eps} | \LDstatic{\rho}{\eps}) &= \Exp{\SBstatic{\rho}{\eps}{}}[c(X,Y,\eps)]-\Exp{\LDstatic{\rho}{\eps}{}}[c(X,Y,\eps)].
    \end{align}
To prove \eqref{eq:sym-ent-identity}, from (\ref{eq:LD-density}) and (\ref{eq:SB-density})
    \begin{align*}
        \frac{d\LDstatic{\rho}{\eps}{}}{d\SBstatic{\rho}{\eps}{}} = \frac{1}{\sqrt{\rho(x)\rho(y)}}\exp\left(-\frac{1}{\eps}f_{\eps}(x)-\frac{1}{\eps}f_{\eps}(y)-c(x,y,\eps)\right).
    \end{align*}
As both $\SBstatic{\rho}{\eps}{}, \LDstatic{\rho}{\eps}{} \in \Pi(\rho,\rho)$, we have that
\begin{align*}
    H(\LDstatic{\rho}{\eps}{} | \SBstatic{\rho}{\eps}) &= \Exp{\LDstatic{\rho}{\eps}{}}\left[\log \frac{d\LDstatic{\rho}{\eps}{}}{d\SBstatic{\rho}{\eps}{}}\right] = - \frac{2}{\eps}\Exp{\rho}[f_{\eps}(X)] -\Ent(\rho) - \Exp{\LDstatic{\rho}{\eps}{}}[c(X,Y,\eps)], \\
    H(\SBstatic{\rho}{\eps}|\LDstatic{\rho}{\eps}{}) &= \Exp{\SBstatic{\rho}{\eps}}\left[-\log \frac{d\LDstatic{\rho}{\eps}{}}{d\SBstatic{\rho}{\eps}{}}\right] = \frac{2}{\eps}\Exp{\rho}[f_{\eps}(X)] + \Ent(\rho) + \Exp{\SBstatic{\rho}{\eps}{}}[c(X,Y,\eps)].
\end{align*}
Equation \eqref{eq:sym-ent-identity} follows by adding these expressions together. 

\noindent\textbf{Step 2.} Recall $R_{\eps}$ as used in \eqref{eq:dynSB-defn} and let
\begin{align}\label{eq:SB-LD-remainder}
    R(x,y,\eps) &:= -\log\left(\Exp{R_{\eps}}\left[\exp\left(-\int_{0}^{1} \eps \left(\cU(\omega_s)-\cU(\omega_0)\right)ds \right)\middle| \omega_0 = x, \omega_1 = y\right]\right).
\end{align}
Then we claim that 
\begin{align}\label{eq:bound-remainder}
    H(\LDstatic{\rho}{\eps}{} | \SBstatic{\rho}{\eps})+H(\SBstatic{\rho}{\eps} | \LDstatic{\rho}{\eps}) &\leq \Exp{\SBstatic{\rho}{\eps}{}}[R(X,Y,\eps)].
\end{align}

Recall that $W_{x}$ is Wiener measure starting from $x$, and
\begin{align*}
    c(x,y,\eps) &= - \log\left(\Exp{W_{x}}\left[\exp\left(-\int_{0}^{\eps} \cU(\omega_s)ds\right) \middle|  \; \omega_\eps = y \right]\right) \\
        &= \eps \cU(x) - \log\left(\Exp{W_{x}}\left[\exp\left(-\int_{0}^{\eps} \left(\cU(\omega_s)-\cU(\omega_0)\right)ds\right)\middle|\; \omega_{\eps} = y\right]\right).
\end{align*}
For the integral inside the $\exp\left(\cdot\right)$, we rescale $s$ from $[0,\eps]$ to $[0,1]$. By Brownian scaling the law of $(\omega_{\eps s}, s \geq 0)$ under $W_{x}$ is equal to the law of $(\omega_s, s \geq 0)$ under $R_{\eps}$ with initial condition $x$, which gives
    $c(x,y,\eps) = \eps \cU(x) + R(x,y,\eps)$.
As $\SBstatic{\rho}{\eps}{} \in \Pi(\rho,\rho)$, thanks to an integration by parts, 
\begin{align*}
    \Exp{\SBstatic{\rho}{\eps}{}}[\cU(X)] &= \int_{\R^d} \cU(x)\rho(x)dx= \frac{1}{8}I(\rho)-\int_{\R^d} \frac{1}{4} \Delta g(x) e^{-g(x)}dx \\
    &= \frac{1}{8}I(\rho) - \int_{\R^d} \frac{1}{4} \|\nabla g(x)\|^{2} e^{-g(x)}dx  = -\frac{1}{8}I(\rho). 
\end{align*}
Thus from \eqref{eq:sym-ent-identity} observe
\begin{align*}
    \Exp{\SBstatic{\rho}{\eps}{}}[c(X,Y,\eps)]-\Exp{\LDstatic{\rho}{\eps}{}}[c(X,Y,\eps)] &\leq -\frac{\eps}{8}I(\rho) + \Exp{\SBstatic{\rho}{\eps}}[R(X,Y,\eps)]-\Exp{\LDstatic{\rho}{\eps}{}}[c(X,Y,\eps)].
\end{align*}
Thus, to prove \eqref{eq:bound-remainder} it will suffice to show
\begin{align}\label{eq:ld-c-upperbdd}
    \Exp{\LDstatic{\rho}{\eps}{}}[-c(X,Y,\eps)] \leq \frac{\eps}{8}I(\rho).
\end{align}
Observe that
\begin{align*}
    \frac{d\LDstatic{\rho}{\eps}{}}{d\BMstatic{\rho}{\eps}{}} &= \frac{\sqrt{\rho(y)}}{\sqrt{\rho(x)}}\exp(-c(x,y,\eps)).
\end{align*}
Thus, as $\LDstatic{\rho}{\eps} \in \Pi(\rho,\rho)$ and $-\infty < \Ent(\rho) < +\infty$, 
\begin{align}\label{eq:rel-ent-ld-bm}
    H(\LDstatic{\rho}{\eps}{}|\BMstatic{\rho}{\eps}{}) &= \Exp{\LDstatic{\rho}{\eps}{}}\left[\log \frac{d\LDstatic{\rho}{\eps}{}}{d\BMstatic{\rho}{\eps}{}}\right] =\Exp{\LDstatic{\rho}{\eps}{}}\left[-c(X,Y,\eps)\right]. 
\end{align}
On $C^{d}[0,\eps]$ let $W_{\rho}$ denote the law of Brownian motion with initial distribution $\rho$. By a well-known formula (see \cite[Equation (5.1)]{DGW} for instance)
\begin{align*}
    H(Q_x|W_x) &= \frac{1}{2} \Exp{Q_x}\left[\int_{0}^{\eps} \|-\frac{1}{2}\nabla g(\omega_s)\|^2 ds\right] = \frac{1}{8}\int_0^{\eps} \Exp{Q_x}[\|\nabla g(\omega_s)\|^2]ds. 
\end{align*}
It then follows from the factorization of relative entropy that
\begin{align*}
    H(Q|W_{\rho}) &= \int H(Q_x|W_x)\rho(x)dx = \frac{1}{8}\int \left(\int_0^{\eps} \Exp{Q_x}\left[\|\nabla g(\omega_s)\|^2\right]ds \right)\rho(x)dx \\
    &= \frac{1}{8}\int_{0}^{\eps}E_{Q}\left[\|\nabla g(\omega_s)\|^2\right]ds = \frac{\eps}{8}I(\rho). 
\end{align*}
As $\LDstatic{\rho}{\eps} = (\pi_0,\pi_{\eps})_{\#}Q$, \eqref{eq:ld-c-upperbdd} follows by \eqref{eq:rel-ent-ld-bm} and the information processing inequality (see \cite[Lemma 9.4.5]{ambrosio2005gradient} for instance).

\noindent\textbf{Step 3.} Show $\Exp{\SBstatic{\rho}{\eps}{}}[R(X,Y,\eps)]$ is bounded above by the RHS in \eqref{eq:rel-ent-ld-sb}. 

To start, apply Jensen's inequality to interchange the $-\log$ and $\Exp{R_{\eps}}[\; \cdot \; |\;\omega_0 = x, \omega_1= y]$ in the expression for $R(x,y,\eps)$ in \eqref{eq:SB-LD-remainder} to obtain
\begin{align*}
    R(x,y,\eps) &\leq \eps\Exp{R_{\eps}}\left[\int_{0}^{1} \left(\cU(\omega_s)-\cU(\omega_0)\right)ds \middle| \; \omega_0 = x,\omega_1 = y\right].
\end{align*}
Since the dynamic $\Schro$ bridge, $\SBdynam{\rho}{\eps}$ defined in \eqref{eq:dynSB-defn}, disintegrates as the Brownian bridge once its endpoints are fixed \cite[Proposition 2.3]{schroLeonard13}, we have that
\begin{align}\label{eq:rxyeps-bound}
    \Exp{\SBstatic{\rho}{\eps}{}}[R(X,Y,\eps)] &\leq \eps \Exp{\SBstatic{\rho}{\eps}{}}\left[\Exp{R_{\eps}}\left[\int_{0}^{1} \left(\cU(\omega_s)-\cU(\omega_0)\right)ds \middle| \; \omega_0 = X,\omega_1 = Y\right]\right]\\
    &= \eps \int_{0}^{1} \Exp{\SBdynam{\rho}{\eps}}\left[\cU(\omega_s)-\cU(\omega_0)\right]ds,
\end{align}
where the last equality follows from Fubini's Theorem. 

Recall that $(\rho_{t}^{\eps},v_{t}^{\eps})_{t \in [0,1]}$ is a weak solution for the continuity equation \eqref{eq:continuity-equation} if for all $\psi \in C_{c}^{1}(\mathbb{R}^{d})$ and $t \in [0,1]$,
$\frac{d}{dt} \Exp{\rho_{t}^{\eps}}[\psi] = \Exp{\rho_{t}^{\eps}}[\langle v_{t}^{\eps}, \nabla \psi \rangle]$.

Integrating in time over $[0,s]$ for $s \in (0,1]$ gives
\begin{align*}
    \Exp{\rho_{s}^{\eps}}[\psi] - \Exp{\rho_{0}^{\eps}}[\psi] &= \int_{0}^{s} \Exp{\rho_{u}^{\eps}}[\langle v_{u}^{\eps}, \nabla \psi \rangle] du.
\end{align*}
We now use a density argument to extend the above identity to
\begin{align}\label{eq:cont-eqn-for-U}
    \Exp{\SBdynam{\rho}{\eps}}[\cU(\omega_s)-\cU(\omega_0)] = \Exp{\rho_{s}^{\eps}}[\cU] - \Exp{\rho_{0}^{\eps}}[\cU] &= \int_{0}^{s} \Exp{\rho_{u}^{\eps}}[\langle v_{u}^{\eps}, \nabla \cU \rangle]du.
\end{align}
For $R > 0$ let $\chi_R: \mathbb{R}^{d} \to \mathbb{R}$ be such that $\chi_R \in C_{c}^{\infty}(\mathbb{R}^{d})$, $0 \leq \chi_R \leq 1$, $\left.\chi_{R}\right|_{B(0,R)} \equiv 1$, and $\|\nabla \chi_R\|_{\infty} \leq 2$. As $\chi_{R} \cU \in C_{c}^{1}(\mathbb{R}^{d})$,
\begin{align*}
    &\Exp{\SBdynam{\rho}{\eps}}[(\chi_{R}\cU)(\omega_s)-(\chi_{R}\cU)(\omega_0)] = \int_{0}^{s} \Exp{\SBdynam{\rho}{\eps}}[\langle v_{u}^{\eps}(\omega_u), \nabla (\chi_{R}\cU)(\omega_u) \rangle]du \\
    &= \int_{0}^{s} \Exp{\SBdynam{\rho}{\eps}}\left[\langle v_u^{\eps},\nabla \cU \rangle \chi_{R}\right]ds + \int_{0}^{s} \Exp{\SBdynam{\rho}{\eps}}\left[\langle v_u^{\eps}, \nabla \chi_{R}\rangle \cU \right]ds.
\end{align*}
By Lemma \ref{lem:integ-entropic-inter}, $\int_{0}^{1} \Exp{\SBdynam{\rho}{\eps}}[\|\nabla \cU(\omega_s)\|^2]ds < +\infty$. Moreover, it is shown in \eqref{eq:integrated-velocity-bound} that $\int_{0}^{1} \Exp{\SBdynam{\rho}{\eps}}[\|v_s^{\eps}(\omega_s)\|^2] ds < +\infty$. Thus, by the Dominated Convergence Theorem
\begin{align*}
    \lim\limits_{R \uparrow +\infty} \int_{0}^{s} \Exp{\SBdynam{\rho}{\eps}}\left[\langle v_u^{\eps},\nabla \cU \rangle \chi_{R}\right]ds &= \int_{0}^{s} \Exp{\SBdynam{\rho}{\eps}}\left[\langle v_u^{\eps},\nabla \cU \rangle \right]ds.
\end{align*}
As $\abs{\chi_{R}\cU(\omega_s)+\chi_{R}\cU(\omega_0)} \leq \abs{\cU(\omega_s)}+\abs{\cU(\omega_0)}$ and 
\begin{align*}
    \abs{\langle v_u^{\eps}, \nabla \chi_{R}\rangle \cU} \leq 2\|v_{u}^{\eps}\|\abs{\cU} \leq \|v_{u}^{\eps}(\omega_u)\|^2+\abs{\cU(\omega_u)}^2
\end{align*}
and both upper bounds are integrable in their suitable spaces (again by Lemma \ref{lem:integ-entropic-inter} and \eqref{eq:integrated-velocity-bound}), \eqref{eq:cont-eqn-for-U} follows from the dominated convergence theorem by sending $R \uparrow +\infty$. 

Integrating \eqref{eq:cont-eqn-for-U} once more in $s$ over $[0,1]$, multiplying by $\eps$, and plugging in \eqref{eq:cont-eq-entropic-inter} gives the RHS of \eqref{eq:rxyeps-bound} as
\begin{align*}
    \eps \int_{0}^{1} \Exp{\SBdynam{\rho}{\eps}} [\cU(\omega_s)-\cU(\omega_0)] ds &= \eps^2 \int_{0}^{1} \int_{0}^{s}  \Exp{\rho_{u}^{\eps}}\left[\left\langle \frac{1}{\eps}v_{u}^{\eps},\nabla \cU\right\rangle\right]  duds.  
\end{align*}
Apply Cauchy-Schwarz on $L^{2}(C^{d}[0,1] \times [0,1], d\SBdynam{\rho}{\eps}\otimes du)$ to obtain for each $s \in [0,1]$
\begin{align*}
    &\abs{-\int_{0}^{s}\Exp{\SBdynam{\rho}{\eps}}\left[\left\langle \frac{1}{\eps} v_{u}^{\eps}(\omega_u),\nabla \cU(\omega_u) \right\rangle\right]  du} \leq \int_{0}^{1} \Exp{\SBdynam{\rho}{\eps}}\left[\frac{1}{\eps}\|v_{u}^{\eps}(\omega_u)\|\|\nabla \cU(\omega_u)\|\right] du \\
    &\leq \left(\int_{0}^{1} \frac{1}{\eps^2}\|v_{u}^{\eps}\|_{L^{2}(\rho_{u}^{\eps})}^{2} du \right)^{1/2}\left(\int_{0}^{1} \Exp{\rho_{u}^{\eps}}[\|\nabla \cU\|^2] du \right)^{1/2}.
\end{align*}
Thus, altogether we have that
\begin{align}\label{eq:almost-bound-rxyeps}
    \Exp{\SBstatic{\rho}{\eps}{}}[R(X,Y,\eps)] &\leq \eps^2\left(\int_{0}^{1} \frac{1}{\eps^2}\|v_{t}^{\eps}\|_{L^{2}(\rho_{t}^{\eps})}^{2} dt\right)^{1/2}\left(\int_{0}^{1} \Exp{\rho_{t}^{\eps}}[\|\nabla \cU\|^2] dt \right)^{1/2}.
\end{align}

Recall the entropic Benamou Brenier expression for entropic cost given in \eqref{eq:symm-ent-cost-BB}. By the optimality of \eqref{eq:cont-eq-entropic-inter} and suboptimality of the constant curve $\rho_t = \rho$ for all $t \in [0,1]$,
\begin{align*}
    H(\SBstatic{\rho}{\eps}|\BMstatic{\rho}{\eps}) &= \frac{1}{2\eps} \int_{0}^{1} \|v_{t}^{\eps}\|_{L^{2}(\rho_t^{\eps})}^{2} dt + \frac{\eps}{8}\int_{0}^{1} I(\rho_{t}^{\eps})dt \leq \frac{\eps}{8}I(\rho).
\end{align*}
Rearranging terms gives
\begin{align}\label{eq:integrated-velocity-bound}
    \frac{1}{\eps^2} \int_{0}^{1} \|v_{t}^{\eps}\|_{L^{2}(\rho_t^\eps)}^2 dt \leq \frac{1}{4}\left(I(\rho)-\int_{0}^{1} I(\rho_t^{\eps}) dt \right).
\end{align}

With this inequality, \eqref{eq:almost-bound-rxyeps} becomes the desired
\begin{align*}
    \Exp{\SBstatic{\rho}{\eps}{}}[R(X,Y,\eps)] &\leq \frac{1}{2}\eps^2\left(I(\rho) - \int_{0}^{1} I(\rho_{t}^{\eps})dt\right)^{1/2}\left(\int_{0}^{1} \Exp{\rho_{t}^{\eps}}[\|\nabla \cU\|^2] dt \right)^{1/2}.
\end{align*}
Under Assumption \ref{assumption:LD_SB_relative_entropy} (2), Lemma \ref{lem:integ-entropic-inter} applies and gives that the rightmost constant has an upper bound once we consider $\eps \in (0,\eps_{0})$ for some $\eps_0 > 0$. Thus, the nonnegativity of the LHS of \eqref{eq:rel-ent-ld-sb} and Lemma \ref{lem:integrated-fisher-continuity} gives the stated convergence as $\eps \downarrow 0$. 
\end{proof}

\textbf{Integrated Fisher Information.} Under additional assumptions the integrated Fisher information along the entropic interpolation in \eqref{eq:rel-ent-ld-sb} can be replaced with a potentially more manageable quantity. In certain settings (for instance, when $\rho \in C_{c}^{\infty}(\mathbb{R}^{d})$ with finite entropy and Fisher information \cite[Lemma 3.2]{glrt-hwi-20}), there is a conserved quantity along the entropic interpolation called the energy, which we denote $\mathcal{E}_{\eps}(\rho)$. With $(v_{t}^{\eps},\rho_{t}^{\eps})_{t \in [0,1]}$ as defined in \eqref{eq:cont-eq-entropic-inter}, for any $t \in [0,1]$ 
\begin{align}\label{eq:energy-along-entropic-inter}
    \mathcal{E}_{\eps}(\rho) := \frac{1}{2\eps^2}\|v_{t}^{\eps}\|_{L^{2}(\rho_{t}^{\eps})}^{2} - \frac{1}{8}I(\rho_{t}^{\eps}).
\end{align}
Now, evaluate \eqref{eq:energy-along-entropic-inter} at $t = 1/2$. Then $v_{1/2}^{\eps} = 0$ and $\mathcal{E}_{\eps}(\rho) = - \frac{1}{8}I(\rho_{1/2}^{\eps})$. On the other hand, integrate \eqref{eq:energy-along-entropic-inter} over $t \in [0,1]$, this also gives $\mathcal{E}_{\eps}(\rho)$ and thus
\begin{align*}
    -\frac{1}{8}I(\rho_{1/2}^{\eps}) &= \frac{1}{2\eps^2}\int_{0}^{1}\|v_t^\eps\|_{L^{2}(\rho_{t}^{\eps})}^{2} dt - \frac{1}{8}\int_{0}^{1} I(\rho_{t}^{\eps}) dt.
\end{align*}
This in turn implies
\begin{align*}
    \frac{1}{\eps^2}\int_{0}^{1}\|v_t^\eps\|_{L^{2}(\rho_{t}^{\eps})}^{2} dt &= \frac{1}{4}\left(\int_{0}^{1} I(\rho_{t}^{\eps}) dt-I(\rho_{1/2}^{\eps})\right) \leq \frac{1}{4}\left(I(\rho) - I(\rho_{1/2}^{\eps})\right).
\end{align*}
Thus, assuming the conservation of energy holds under Assumption \ref{assumption:LD_SB_relative_entropy}, replacing the bound in \eqref{eq:integrated-velocity-bound} with the RHS of the above inequality rewrites Theorem \ref{thm:ld-sb-general-g} as
\begin{align}\label{eq:rel-ent-sb-ld-midpoint}
    H(\LDstatic{\rho}{\eps} | \SBstatic{\rho}{\eps})+H(\SBstatic{\rho}{\eps} | \LDstatic{\rho}{\eps}) \leq \frac{1}{2}\eps^2 \left(I(\rho) - I(\rho_{1/2}^{\eps})\right)^{1/2}\left(\int_{0}^{1} \Exp{\rho_t^\eps}\|\nabla \cU\|^2 dt\right)^{1/2}.
\end{align}
Alternatively, \eqref{eq:rel-ent-sb-ld-midpoint} also holds when $\rho$ is a univariate Gaussian, as explicit computations of the entropic interpolation in this case from \cite{gentil2017analogy} demonstrate that $I(\rho_{t}^{\eps})$ is minimized at $t = 1/2$.

\textbf{$\Schro$ Cost Expansion.}
In the course of proving Theorem \ref{thm:ld-sb-general-g}, we recover a result in alignment with the expansion of $\Schro$ cost about $\eps = 0$ developed in \cite[Theorem 1.6]{conforti21deriv}. Under assumptions on bounded $\rho$ with compact support and finite entropy, \cite[Theorem 1.6]{conforti21deriv} states
\begin{align}\label{eq:ct-thm1-6}
    H(\SBstatic{\rho}{\eps}|\BMstatic{\rho}{\eps}) &= \frac{\eps}{8}I(\rho) + o(\eps). 
\end{align}
Theorem \ref{thm:ld-sb-general-g} allows the next term in the expansion to be ascertained.
That is, in the same-marginal case (and under different regularity assumptions) the second order term in the $\Schro$ expansion is zero. 
\begin{corollary}\label{cor:CTextension}
    Let $\rho = e^{-g}$ satisfy Assumption \ref{assumption:LD_SB_relative_entropy}. Then
    \begin{align}\label{eq:CTextension}
    H(\SBstatic{\rho}{\eps}|\BMstatic{\rho}{\eps}) &= \frac{\eps}{8}I(\rho) - o(\eps^2). 
\end{align}
\end{corollary}
\begin{proof}
    From the Pythagorean Theorem of relative entropy \cite[Theorem 2.2]{csiszar75idiv},
    \[H(\LDstatic{\rho}{\eps}|\BMstatic{\rho}{\eps}) \geq H(\LDstatic{\rho}{\eps}|\SBstatic{\rho}{\eps}) + H(\SBstatic{\rho}{\eps}|\BMstatic{\rho}{\eps}).\]
Recall from Step 2 of Theorem \ref{thm:ld-sb-general-g} that $H(\LDstatic{\rho}{\eps}|\BMstatic{\rho}{\eps}) \leq \frac{\eps}{8}I(\rho)$. Thus, the expansion follows by the bound Theorem \ref{thm:ld-sb-general-g} provides for $H(\LDstatic{\rho}{\eps}|\SBstatic{\rho}{\eps})$. 
\end{proof}

\section{Small Temperature Approximation to Entropic Brenier Map}\label{sec:one-step}
This section develops applications of the same-marginal $\Schro$ bridge approximation developed in Theorem \ref{thm:ld-sb-general-g}. We present and discuss the Theorems in the present section. The proofs of the Theorems are closely related and given in Section \ref{sec:technical-lemmas}.

The major result of this section is Theorem \ref{thm:score-function-approx}, which establishes the following expansion of the same-marginal barycentric projection (entropic Brenier map) about the Brenier map, $\Id$, in $\mathbf{L}^2(\rho)$:
\begin{align*}
    \BPbase{\rho}{\eps} &= \Id + \frac{\eps}{2}\nabla \log \rho + o(\eps).
\end{align*}
The precise statement is given below. 
\begin{theorem}\label{thm:score-function-approx}
Let $\rho = e^{-g}$ satisfy Assumption \ref{assumption:LD_SB_relative_entropy}. Additionally, assume that $g$ is semiconvex and $\Exp{\rho}\|\nabla^{2} g\|_{HS}^{2} < +\infty$. Then for all $\eps > 0$ small enough, there is a constant $K > 0$ such that
    \begin{align*}
        \norm{\BPbase{\rho}{\eps}-\left(\Id + \frac{\eps}{2}\nabla \log \rho\right)}_{L^{2}(\rho)} \leq K \eps\left[ \left(I(\rho)-\int_{0}^{1}I(\rho^{\eps}_{s})ds\right)^{1/4} + \sqrt{\eps}\right].
    \end{align*}
    Thus, in $\mathbf{L}^2(\rho)$ the following convergences hold, with $f_{\eps}$ as defined in \eqref{eq:SB-density}:
    \begin{align*}
        \lim\limits_{\eps \downarrow 0} \frac{1}{\eps}\left(\BPbase{\rho}{\eps}-\Id\right) = \frac{1}{2}\nabla \log \rho &\textup{ and }
        \lim\limits_{\eps \downarrow 0} \frac{1}{\eps}\nabla f_{\eps} = -\frac{1}{2}\nabla \log \rho.
    \end{align*}
    As such, it holds that $\lim\limits_{\eps \downarrow 0} \frac{1}{\eps^{2}}\|\BPbase{\rho}{\eps}-\Id\|_{\mathbf{L}^2(\rho)}^{2} = \frac{1}{4}I(\rho)$. 
\end{theorem}
This is in accordance with \cite[Theorem 1]{sander_22}, but now with fully supported densities. 

Theorem \ref{thm:score-function-approx} follows from a more general result that we now present in Theorem \ref{thm:sb-generator-approx}. We argue that for sufficiently regular $\xi: \mathbb{R}^{d} \to \mathbb{R}$,
\begin{align*}
    \Exp{\SBstatic{\rho}{\eps}}[\xi(Y)|X=x] = \xi(x) + \eps L \xi (x) + o(\eps),
\end{align*}
where $L$ is the generator of the Langevin diffusion defined in \eqref{eq:LD-generator}. In other words, the conditional density of one coordinate of the $\Schro$ bridge given the other is approximated by the transition density of the Langevin diffusion for small $\eps$. A possible application of this result is in the simulation of the conditional density of the small temperature $\Schro$ bridge.

\begin{theorem}\label{thm:sb-generator-approx}
    Let $\rho = e^{-g}$ satisfying the assumptions of Theorem \ref{thm:score-function-approx}, and let $L$ be as defined in \eqref{eq:LD-generator}. For all Lipschitz $\xi \in \mathcal{D}(L)$, the following convergence holds in $\mathbf{L}^{2}(\rho)$
    \begin{align*}
        \lim\limits_{\eps \downarrow 0}\frac{1}{\eps}\left(\Exp{\SBstatic{\rho}{\eps}}[\xi(Y)|X=x]-\xi(x)\right) &= L \xi(x).
    \end{align*}
\end{theorem}

We emphasize that this family of integral operators $(\Exp{\SBstatic{\rho}{\eps}}[\cdot|X],\eps > 0)$ does not form a Markov semigroup. Nevertheless, we show that it does admit a derivative at zero which is the generator of the Langevin diffusion. Hence, for small $\eps > 0$, this family of integral operators behaves approximately like the Langevin semigroup, a classical object studied in both probability and PDEs. As a consequence, it uncovers an approximate Markov structure lurking in the small temperature same-marginal $\Schro$ bridge. Namely, if we compute the same marginal $\Schro$ bridge with temperature $\eps/2$ and concatenate it with itself, we obtain a close approximation to the $\Schro$ bridge with temperature $\eps$.

For the last Theorem of this section, we produce a generalization of Theorem \ref{thm:score-function-approx}. Specifically, Theorem \ref{thm:one_step_convergence} demonstrates that the barycentric projection can approximate arbitrary elements of the tangent space at $\rho \in \cP_{2}^{ac}(\mathbb{R}^{d})$. We now introduce in full generality the ideas discussed in Section \ref{sec:introduction}.

Recall that the tangent space at $\rho$, denoted by $\Tan{\rho}$, is defined in \cite[Definition 8.4.1]{ambrosio2005gradient} by
\begin{align*}
    \Tan{\rho} := \overline{\{\nabla \varphi: \varphi \in C_{c}^{\infty}(\mathbb{R}^{d})\}}^{L^{2}(\rho)}.
\end{align*}
Now, consider a vector field $v=\nabla u \in \Tan{\rho}$, such that the surrogate measure \eqref{eq:surrogate} is defined. Our objective in this section is to lay down conditions under which the following holds
\begin{equation}\label{eq:short_thm_2}
\lim_{\eps \downarrow 0}\frac{1}{\eps}\Was{2}(\SB{\rho}{\eps}{},(\Id + \eps v)_{\#}\rho)=0.
\end{equation}

In fact, we will prove the above theorem in a somewhat more general form when the tangent vector field may not even be of a gradient form. The general setting is as follows. Consider a collection of vectors fields of gradient type $(v_{\eps} = \nabla \psi_{\eps}, \eps > 0)$,
we wish to tightly approximate the pushforwards $(\Id+\eps v_{\eps})_{\#}\rho$ with $\SB{\rho}{\eps}{}$. As a shorthand, we write
\begin{align}\label{eq:S-eps}
    \opt{\eps}{1}{\rho} := (\Id + \eps v_{\eps})_{\#}\rho.
\end{align}
Suppose that for each $\eps > 0$ there exists $\theta_\eps \in \mathbb{R} \setminus \{0\}$ such that the surrogate measure defined in \eqref{eq:surrogate} exists. Set
\begin{align}\label{eq:simga-eps-onestep}
    \sigma_{\eps}(x) &= \exp\left(2\theta_{\eps} \psi_{\eps}(x)- \Lambda_0(\theta_{\eps})\right), \quad \text{where}\quad \Lambda_0(\theta_{\eps}):=\log \int e^{2\theta \psi_{\eps}(y)}dy.
\end{align}
Redefine the SB step as
\begin{align}\label{sb-step-general}
    \tag{SB} \SB{\rho}{\eps}{} &= \left(\left(1-\theta_{\eps}^{-1}\right)\Id + \theta_{\eps}^{-1} \BPbase{\sigma_\eps}{\eps}\right)_{\#}\rho,
\end{align}
where $\BPbase{\sigma_\eps}{\eps}$ is the barycentric projection defined in \eqref{eq:bary-proj}.

In Theorem \ref{thm:one_step_convergence} we show that, under appropriate assumptions, 
\begin{equation}\label{eq:short_thm_2-general}
\lim_{\eps \downarrow 0}\frac{1}{\eps}\Was{2}(\SB{\rho}{\eps}{},\opt{\eps}{1}{\rho})=0.
\end{equation} 

There are multiple reasons why one might be interested in this general setup. Firstly, we know that, for any $v \in \Tan{\rho}$, there exists a sequence of tangent elements of the type $v_\eps=\nabla \psi_\eps$ such that $\lim_{\eps\rightarrow 0+} v_\eps=v$ in $\mathbf{L}^2(\rho)$. In particular, by an obvious coupling, 
\begin{equation}\label{eq:ltwoapprox}
\limsup_{\eps \rightarrow 0+}\frac{1}{\eps} \Was{2}( (\Id + \eps v_{\eps})_{\#}\rho ,(\Id + \eps v)_{\#}\rho)\le \lim_{\eps \rightarrow 0+}\norm{v_\eps - v}_{\mathbf{L}^2(\rho)}=0.
\end{equation}
But, of course, if $v$ is not of the gradient form, one cannot construct a surrogate measure as in \eqref{eq:surrogate} to run the Schr\"odinger bridge scheme. A natural remedy is to show that \eqref{eq:short_thm_2-general} holds under appropriate assumptions. Then, combined with \eqref{eq:ltwoapprox} we recover \eqref{eq:short_thm_2} even when $v$ is not a gradient. 

There is one more reason why one might use the above generalized scheme. The more natural discretization scheme of Wasserstein gradient flows is the implicit Euler or the JKO scheme. In this case, however, the vector field $v_\eps$, which is the Kantorovich potential between two successive steps of the JKO scheme, changes with $\eps$. This necessitates our generalization. Hence, we continue with \eqref{eq:simga-eps-onestep} and the generalized \eqref{sb-step-general} step. 

The proof of \eqref{eq:short_thm_2-general} relies on the close approximation of the Langevin diffusion to the $\Schro$ bridge provided in Theorem \ref{thm:ld-sb-general-g}. To harness this approximation, we introduce an intermediary scheme based on the Langevin diffusion, which we denote $\LD{\cdot}{\eps}{}$ and define in the following manner. Let $\eps > 0$ and consider the symmetric Langevin diffusion $\left( X_t,\; 0\le t \le \eps\right)$ with stationary density $\sigma_{\eps}$ as defined in \eqref{eq:simga-eps-onestep}. The drift function for this diffusion is $x \mapsto \frac{1}{2}\nabla \log \sigma_{\eps}(x)= \theta_{\eps}\nabla \psi_\eps(x)=\theta_{\eps}v_{\eps}(x)$. In analogy with the definition of $\BPbase{\sigma_{\eps}}{\eps}(\cdot)$ in \eqref{eq:bary-proj}, define $\LBPbase{\sigma_{\eps}}{\eps}(\cdot)$ by
\begin{align}\label{eq:lpbase-defn}
    \LBPbase{\sigma_{\eps}}{\eps}(x) &= \Exp{\LDstatic{\sigma_{\eps}}{\eps}}[Y|X=x].
\end{align}
Next, in analogy with the definition of $\SB{\cdot}{\eps}{}$ in \eqref{sb-step-general}, define the Langevin diffusion update
\begin{align}\label{eq:LD-update}
    \tag{LD} \LD{\rho}{\eps}{} &= \left(\left(1-\theta_\eps^{-1}\right)\Id + \theta_\eps^{-1}\LBPbase{\sigma_{\eps}}{\eps}\right)_{\#}\rho. 
\end{align}
In analogy with Section \ref{sec:sb-ld}, for $\eps > 0$ set
    \begin{equation}\label{eq:defineueps}
        \cU_{\eps} := \frac{1}{8}\|\nabla \log \sigma_{\eps}\|^2 + \frac{1}{4}\Delta \log \sigma_{\eps}.
    \end{equation}
We now make the following assumptions on the collection of densities $(\sigma_\eps, \eps > 0)$ defined in \eqref{eq:simga-eps-onestep}.
\begin{assumption}\label{assumptions:sigma-eps-onestep}
For each $\eps > 0$, $\sigma_\eps$ as defined in \eqref{eq:simga-eps-onestep} satisfies Assumption \ref{assumption:LD_SB_relative_entropy}. Additionally
\begin{itemize}
    \item[(1)] There exists $\sigma_0 \in \cP_{2}^{ac}(\mathbb{R}^{d})$ such that $(\sigma_\eps, \eps > 0)$ converges weakly to $\sigma_0$ and $I(\sigma_\eps) \to I(\sigma_0)$ as $\eps \downarrow 0$. Moreover, Assumption \ref{assumption:LD_SB_relative_entropy} holds uniformly in $\eps$ in the following manner: there exists $\eps_0 > 0$ and $C,D > 0$, $N \geq 1$ such that for all $\eps \in (0,\eps_0)$, $\|\nabla \cU_{\eps}(x)\|^2 \leq C(1+\|x\|^{N})$ and for $X_{\eps} \sim \sigma_{\eps}$, $\|X_{\eps}\|_{\psi_{1}} \leq D$.
    \item[(2)] The $(\sigma_{\eps}, \eps \in (0,\eps_0))$ are uniformly semi log-concave, that is, there exists $\lambda \in \mathbb{R}$ such that for all $\eps \in (0,\eps_0)$, $- \nabla^{2}\log \sigma_{\eps} \geq \lambda \Id$. Additionally, $L := \limsup\limits_{\eps \downarrow 0} \Exp{\sigma_{\eps}}\|\nabla^2 \log \sigma_{\eps}\|_{HS}^{2} < +\infty$, where $\|.\|_{HS}$ is the Hilbert-Schmidt norm.
    \item[(3)] $\chi:=\limsup\limits_{\eps
     \downarrow 0} \|\rho/\sigma_{\eps
     }\|_{\infty} < +\infty$ and $\theta:=\liminf\limits_{\eps \downarrow 0} \abs{\theta_{\eps}} > 0$.
\end{itemize}
\end{assumption}
  
\begin{remark}\label{rmk:theta}
   In the examples we discuss, $\theta_\eps$ is either always $+1$ or always $-1$. Hence, this parameter simply needs to have a positive lower bound (in absolute value) and it does not scale with $\eps$. 
\end{remark}

While these assumptions may appear stringent, they are indeed satisfied by a wide class of examples. In many of these examples, $\sigma_{\eps}$ takes the form $\sigma_{\eps} = \rho e^{V_{\eps}}$. Thus, condition (3) in Assumption \ref{assumptions:sigma-eps-onestep} is satisfied if $\liminf\limits_{\eps \downarrow 0} \inf\limits_{x\in \mathbb{R}^{d}} V_{\eps} > -\infty$. For instance, in the proof of Theorem $\ref{thm:score-function-approx}$, $\sigma_{\eps} = \rho$ for all $\eps > 0$ and thus $V_{\eps} \equiv 0$. When approximating velocity fields coming from the gradient flow of entropy, $\Ent(\cdot)$, it is also the case that $\sigma_\eps = \rho$ for all $\eps > 0$. Similarly, for velocity fields coming from the gradient flow of $H(\cdot | N(0,1))$ from certain starting measures, $V_{\eps}(x) = \frac{1}{2}x^2 + C$, where $C$ is a normalizing constant, for all $\eps > 0$. For more detailed discussion, we refer readers to the examples section of \cite{AHMP} (the preprint in which these results originally appeared). 

We now provide a quantitative statement about the convergence stated in \eqref{eq:short_thm_2-general}, where we take care to note the dependence on constants appearing in Assumption \ref{assumptions:sigma-eps-onestep}. Recall $\opt{\eps}{1}{\cdot}$ as defined in \eqref{eq:S-eps} and $\SB{\cdot}{\eps}{}$ as defined in \eqref{sb-step-general}.

\begin{theorem}\label{thm:one_step_convergence}
Under Assumption \ref{assumptions:sigma-eps-onestep}, there exists a constant $K > 0$ depending on $C,D,N,\lambda,L, \chi,$ $\theta$, such that for all $\eps \in (0,\eps_0)$
\begin{align*}
    \Was{2}(\SB{\rho}{\eps}{},\opt{\eps}{1}{\rho})&\leq
    K \eps\left[ \left(I(\sigma_{\eps})-\int_{0}^{1}I((\sigma_{\eps})^{\eps}_{s})ds\right)^{1/4} + \sqrt{\eps}\right].
\end{align*}  

In particular, $\lim\limits_{\eps \downarrow 0} \eps^{-1}\Was{2}(\SB{\rho}{\eps}{},\opt{\eps}{1}{\rho})=0$. 
\end{theorem}

Hence, from the discussion following \eqref{eq:ltwoapprox} we immediately obtain:

\begin{corollary}\label{lem:general-v-approx}
Let $v \in \Tan{\rho}$, and let $(\nabla \psi_{\eps}, \eps > 0) \subset C^{\infty}(\mathbb{R}^{d}) \cap \Tan{\rho}$ be such that $\lim\limits_{\eps \to 0+} \|\nabla \psi_{\eps}-v\|_{\mathbf{L}^2(\rho)} =0$ and the corresponding surrogate measures $(\sigma_{\eps} > 0)$ satisfy Assumption \ref{assumptions:sigma-eps-onestep}. Then
$    \lim\limits_{\eps \downarrow 0+} \eps^{-1}\Was{2}((\Id+\eps v)_{\#}\rho,\SB{\rho}{\eps}{}) = 0$.
\end{corollary}

\subsection{Technical Lemmas and Proofs}\label{sec:technical-lemmas}
We now prove Theorems \ref{thm:score-function-approx}- \ref{thm:one_step_convergence}. The proofs are very connected and rely on three technical lemmas that we present upfront.

We begin by converting the relative entropy decay established in Theorem \ref{thm:ld-sb-general-g} between the $\Schro$ bridge and Langevin diffusion into a decay rate between their respective barycentric projections in the Wasserstein-2 metric. The essential tool for this conversion is a Talagrand inequality from \cite[Theorem 1.8(3)]{cattiaux-guillin-slc-14}.

\begin{lemma}\label{lem:BP_upperbounds}
Let $\eps > 0$. Under Assumption \ref{assumptions:sigma-eps-onestep}, there exists a constant $\alpha>0$ depending on $\eps$ and $\lambda$ such that
\begin{align*}
    \Was{2}^{2}((\BPbase{\sigma_{\eps}}{\eps})_{\#}\rho,(\LBPbase{\sigma_{\eps}}{\eps})_{\#}\rho) &\leq \alpha \|\rho/\sigma_{\eps}\|_{\infty}H(\SBstatic{\sigma_{\eps}}{\eps}| \LDstatic{\sigma_{\eps}}{\eps}) \\
    \Was{2}^{2}(\SB{\rho}{\eps}{},\LD{\rho}{\eps}{}) &\leq \alpha\theta_{\eps}^{-2} \|\rho/\sigma_{\eps}\|_{\infty}H(\SBstatic{\sigma_{\eps}}{\eps}| \LDstatic{\sigma_{\eps}}{\eps}).
\end{align*}
Moreover, $\alpha$ can be chosen to be increasing in $\eps$. 
\end{lemma}

Recall that $\theta_{\eps} \in \mathbb{R} \setminus \{0\}$ is a constant chosen for integrability reasons outlined in \eqref{eq:simga-eps-onestep}-- it is simply an $O(1)$ term as $\eps \downarrow 0$. See Remark \ref{rmk:theta}.

\begin{proof}[Proof of Lemma \ref{lem:BP_upperbounds}]
Fix $T > 0$ and $\eps \in (0,T)$. To avoid unnecessary subscripts, set $\sigma := \sigma_{\eps}$. To start, we produce a coupling of $\LDstatic{\sigma}{\eps}$ and $\SBstatic{\sigma}{\eps}$ by applying a gluing argument (as in \cite[Lemma 7.6]{villani2021topics}, for example) to the optimal couplings between corresponding transition densities for the Langevin diffusion and $\eps$-static $\Schro$ bridge. From this coupling, we obtain a coupling of the desired barycentric projections of $\LDstatic{\sigma}{\eps}$ and $\SBstatic{\sigma}{\eps}$. We then use the chain rule of relative entropy alongside the Talagrand inequality stated in \cite[Theorem 1.8(3)]{cattiaux-guillin-slc-14} to obtain the inequality stated in the Lemma. 

Let $(q_{t}(\cdot,\cdot),t \geq 0)$ denote the transition densities of the stationary Langevin diffusion associated to $\sigma$. By \cite[Theorem 1.8(3)]{cattiaux-guillin-slc-14}, there is a constant $\alpha>0$ depending on $\lambda$ and $T$ such that following Talagrand inequality holds for each $x \in \mathbb{R}^{d}$,
$\Was{2}^{2}(\eta,q_{\eps}(x,\cdot)) \leq \alpha H(\eta |q_{\eps}(x,\cdot))$, for all $\eta \in \cP(\mathbb{R}^{d})$. Let $(p(x,\cdot), x \in \mathbb{R}^{d})$ denote the conditional densities of the $\eps$-static $\Schro$ bridge with marginals $\sigma$. Construct a triplet of $\mathbb{R}^{d}$-valued random variables $(X,Y,Z)$ in the following way. Let $X \sim \rho$, and let the conditional measures $(Y|X=x,Z|X=x)$ have the law of the quadratic cost optimal coupling between $p(x,\cdot)$ and $q_{\eps}(x,\cdot)$. Denote the law of $(X,Y,Z)$ on $\mathbb{R}^{d} \times \mathbb{R}^{d} \times \mathbb{R}^{d}$ by $\nu$. 
Observe that $(\Exp{v}[Y|X],\Exp{v}[Z|X])$ is a coupling of $(\LBPbase{\sigma_{\eps}}{\eps})_{\#}\rho$ and $(\BPbase{\sigma_{\eps}}{\eps})_{\#}\rho$. Moreover, observe that $((1-\theta_{\eps}^{-1})X+\theta_{\eps}^{-1}\Exp{v}[Y|X],(1-\theta_{\eps}^{-1})X+\theta_{\eps}^{-1}\Exp{v}[Z|X])$ is a coupling of $\SB{\rho}{\eps}{}$ and $\LD{\rho}{\eps}{}$. Thus, to compute an upper bound of $\Was{2}^{2}(\SB{\rho}{\eps}{},\LD{\rho}{\eps}{})$, it is sufficient to compute $\Exp{}\|\Exp{v}[Y|X]-\Exp{v}[Z|X]\|^2$.

We first observe for each $x \in \mathbb{R}^{d}$ that by two applications of the (conditional) Jensen's inequality, the optimality of $(Y|X=x,Z|X=x)$, and the aforementioned Talagrand inequality 
\begin{align*}
    \|\Exp{v}[Y|X=x]-\Exp{v}[Z|X=x]\| 
    &\leq \Was{2}(q_{\eps}(x,\cdot),p(x,\cdot)) \leq \sqrt{\alpha H(p(x,\cdot) |q_{\eps}(x,\cdot))}.
\end{align*}
Hence, using the chain rule of relative entropy and the fact that the Langevin diffusion and $\Schro$ bridge both start from $\sigma$
\begin{align*}
    \Was{2}^{2}&((\BPbase{\sigma}{\eps})_{\#}\rho,(\LBPbase{\sigma}{\eps})_{\#}\rho) \leq \Exp{v}\|\Exp{v}[Y|X=x]-\Exp{v}[Z|X=x]\|^2 \\
    &\leq \int \alpha H(p(x,\cdot)|q_{\eps}(x,\cdot)) \rho(x)dx = \int \alpha H(p(x,\cdot)|q_{\eps}(x,\cdot)) \sigma(x)\frac{\rho(x)}{\sigma(x)}dx\\
    &\leq \alpha \|\rho/\sigma_{\eps}\|_{\infty}H(\SBstatic{\sigma_{\eps}}{\eps}{}|\LDstatic{\sigma_{\eps}}{\eps}{}).
\end{align*}
\end{proof}

Now, we revisit the bound from Theorem \ref{thm:ld-sb-general-g} and apply it to Lemma \ref{lem:BP_upperbounds}.
\begin{lemma}\label{lem:bound-preserved-eps}
    Under Assumption \ref{assumptions:sigma-eps-onestep}, there exists a constant $K>0$ depending on $N$, $C$, and $D$ such that for all $\eps \in (0,\eps_0)$ 
    \begin{align}\label{eq:bound-preserved-eps}
        H(\SBstatic{\sigma_{\eps}}{\eps}|\LDstatic{\sigma_{\eps}}{\eps})+H(\LDstatic{\sigma_{\eps}}{\eps}|\SBstatic{\sigma_{\eps}}{\eps}) \leq \frac{\eps^2}{2}K\left(I(\sigma_{\eps})-\int_{0}^{1} I((\sigma_{\eps
    })_{s}^{\eps}) ds \right)^{1/2}  
    \end{align}
    In particular,  
    $\lim\limits_{\eps \downarrow 0} \eps^{-2}\left(H(\SBstatic{\sigma_{\eps}}{\eps}|\LDstatic{\sigma_{\eps}}{\eps})+H(\LDstatic{\sigma_{\eps}}{\eps}|\SBstatic{\sigma_{\eps}}{\eps})\right) = 0.$
\end{lemma}
\begin{proof}
    Let $\eps \in (0,\eps_0)$ and in analogy with the Section \ref{sec:sb-ld} let $((\sigma_{\eps})_{s}^{\eps},s\in[0,1])$ denote the $\eps$-entropic interpolation from $\sigma_{\eps}$ to itself and recall \eqref{eq:defineueps}.
    By the assumed uniform subexponential bound on the $\sigma_{\eps}$ and uniform polynomial growth of $\cU_{\eps}$, there exists a constant $K > 0$ as specified in the Lemma statement such that     
    \begin{align}\label{eq:upper-bdd-varyingeps}
        \sup\limits_{\eps \in (0,\eps_0),s \in [0,1]}\Exp{(\sigma_{\eps})_{s}^{\eps}}[\|\nabla \cU_{\eps}\|^2] \leq K.
    \end{align}
    Theorem \ref{thm:ld-sb-general-g} and \eqref{eq:upper-bdd-varyingeps} then establish \eqref{eq:bound-preserved-eps} when $\eps \in (0,\eps_0)$.
    It now remains to show that
    \begin{align}\label{eq:varying-eps-cont-fisher}
        \lim\limits_{\eps \downarrow 0} \left[I(\sigma_{\eps})-\int_{0}^{1} I((\sigma_{\eps
    })_{s}^{\eps}) ds \right] = 0.
    \end{align}
    First, we claim that $((\sigma_{\eps})_{s}^{\eps}, \eps > 0)$ converges weakly to $ \sigma_{0}$ as $\eps \downarrow 0$ for all $s \in [0,1]$. As $(\sigma_{\eps}, \eps > 0)$ converges weakly to $\sigma_0$, the collection $\{\sigma_0\} \cup (\sigma_{\eps}, \eps > 0)$ is tight. This implies that $(\SBstatic{\sigma_{\eps}}{\eps}, \eps > 0)$ is tight. Let $\pi^{*}$ denote a weak subsequential limit-- we will not denote the subsequence. It is clear that $\pi^{*} \in \Pi(\sigma_0,\sigma_0)$. In the terminology of \cite{bgn-eot-gld}, the support of each $\SBstatic{\sigma_{\eps}}{\eps}$ is $(c,\eps)$-cyclically invariant, where $c(x,y)= \frac{1}{2}\|x-y\|^2$. As $\pi^*$ is a weak subsequential limit of the $(\SBstatic{\sigma_{\eps}}{\eps}, \eps > 0)$, the exact same argument presented in \cite[Lemma 3.1, Lemma 3.2]{bgn-eot-gld}, modified only superficially to allow for the different marginal in each $\eps$, establishes that $\pi^*$ is $c$-cyclically monotone. That is, $\pi^* = \SBstatic{\sigma_0}{0}$ is the quadratic cost optimal transport plan from $\sigma_0$ to itself. Moreover, this argument shows that along any subsequence of $\eps$ decreasing to $0$, there is a further subsequence converging to $\SBstatic{\sigma_0}{0}$. Hence, the limit holds without passing to a subsequence. 
 
    Construct a sequence of random variables $((X_{\eps},Y_{\eps}), \eps \geq 0)$ with $\text{Law}(X_{\eps},Y_{\eps}) = \SBstatic{\sigma_{\eps}}{\eps}$, and let $Z$ be a standard normal random variable independent to the collection. Fix $s \in [0,1]$, and recall that $\text{Law}(sX_{\eps}+(1-s)Y_{\eps}+\sqrt{\eps s(1-s)}Z) = (\sigma_{\eps})_{s}^{\eps}$. The Continuous Mapping Theorem gives that $sX_{\eps}+(1-s)Y_{\eps}+\sqrt{\eps s(1-s)}Z$ converges in distribution to $sX_{0}+(1-s)Y_{0}$ as $\eps \downarrow 0$. 
    As $\text{Law}(sX_0+(1-s)Y_0) = \sigma_0$, this establishes the weak convergence of $(\sigma_{\eps})_{s}^{\eps}$ to $\sigma_{0}$ as $\eps \downarrow 0$. By the lower semicontinuity of Fisher information with respect to weak convergence \cite[Proposition 14.2]{bobkov-fisher-22} and Fatou's Lemma
    $
        I(\sigma_0) \leq \int_{0}^{1} \liminf\limits_{\eps \downarrow 0} I((\sigma_{\eps})_{s}^{\eps}) ds \leq \liminf\limits_{\eps \downarrow 0} \int_{0}^{1} I((\sigma_{\eps})_{s}^{\eps}) ds.
    $
    Thus, \eqref{eq:varying-eps-cont-fisher} holds as Assumption \ref{assumptions:sigma-eps-onestep}(1) gives $\lim\limits_{\eps \downarrow 0} I(\sigma_{\eps}) = I(\sigma_0)$. 
\end{proof}

As the last piece of the proof, we show in the following lemma below that $\LD{\rho}{\eps}{}$ is an $o(\eps)$ approximation of the first step of the Euler iteration $\opt{\eps}{1}{\rho}$. 

\begin{lemma}[]\label{lem:LD_BP_covergence}
For $\rho = e^{-g}$ satisfying Assumption \ref{assumptions:sigma-eps-onestep}, 
\[
\Was{2}^2(\opt{\eps}{1}{\rho}, \LD{\rho}{\eps}{}) \leq 2\theta_{\eps}^{-2}\|\rho/\sigma_{\eps}\|_{\infty}\eps^3 \Exp{\sigma_{\eps}}\|\nabla^2 \log \sigma_{\eps}\|^2_{HS}\,.
\]
\end{lemma}
\begin{proof}
The proof follows by producing an intuitive coupling of $\opt{\eps}{1}{\rho}$ and $\LD{\rho}{\eps}{}$ (defined in \eqref{eq:S-eps} and \eqref{eq:LD-update}, respectively) and then applying a semigroup property stated in \cite[Equation (4.2.3)]{bgl-markov}. Let $(G_{s}^{\sigma_{\eps}}, s \geq 0)$ be the semigroup associated to the Langevin diffusion corresponding to $\sigma_{\eps}$, then for all $\eps > 0$, $s \geq 0$, and $f: \mathbb{R}^{d} \to \mathbb{R}$ in the domain of the Dirichlet form 
\begin{align}\label{eq:semigrp-bdd}
    \|G_{s}^{\sigma_{\eps}}f - f\|^{2}_{L^{2}(\sigma_{\eps})} \leq 2s\|\nabla f\|_{L^{2}(\sigma_{\eps})}^{2}. 
\end{align}
Let $X \sim \rho$, then by Dynkin's formula, \eqref{eq:lpbase-defn} becomes
\begin{align*}
    \LBPbase{\rho}{\eps}(X) &= X+\int_{0}^{\eps}G^{\sigma_{\eps}}_{s}(\theta_{\eps}\nabla \psi_{\eps})(X)ds,
\end{align*}
and thus
    \begin{align*}
    \LD{\rho}{\eps}{} &= \text{Law}\left((1-\theta_{\eps}^{-1})X+\theta_{\eps}^{-1}\left(X+\int_{0}^{\eps}G^{\sigma_{\eps}}_{s}(\theta_{\eps}\nabla \psi_{\eps})(X)ds\right)\right)\\
    &= \text{Law}\left(X+\int_{0}^{\eps}G^{\sigma_{\eps}}_{s}(\nabla \psi_{\eps})(X)ds\right). 
    \end{align*}
Also by definition
    $\opt{\eps}{1}{\rho} = \text{Law}\left(X + \eps \nabla \psi_{\eps}(X)\right)$. 
    Thus,
    \begin{align*}
        &\Was{2}^{2}(\LD{\rho}{\eps}{},\opt{\eps}{1}{\rho}) \leq \Exp{} \left\|\left(X+\int_{0}^{\eps}G^{\sigma_{\eps}}_{s}(\nabla \psi_{\eps})(X)ds\right)-\left(X + \eps \nabla \psi_{\eps}(X))\right)\right\|^{2} \\
        &= \Exp{}\left\|\int_{0}^{\eps} \left(G_{s}^{\sigma_{\eps}}(\nabla \psi_{\eps})(X)-\nabla \psi_{\eps}(X) \right)ds \right\|^2 \\
        &\leq \eps \int_{0}^{\eps} \Exp{}\|G_{s}^{\sigma_{\eps}}(\nabla \psi_{\eps})(X)-\nabla \psi_{\eps}(X)\|^2 ds = \eps^2 \int_{0}^{1} \|G_{\eps s}^{\sigma_{\eps}}(\nabla \psi_{\eps})-\nabla \psi_{\eps}\|_{L^{2}(\rho)}^2 ds,
    \end{align*}
    where the penultimate inequality follows from Jensen's inequality. Then, perform a change of measure and apply \eqref{eq:semigrp-bdd} to obtain
    \begin{align*}
        \Was{2}^{2}(\LD{\rho}{\eps}{},\opt{\eps}{1}{\rho}) &\leq 2\|\rho/\sigma_{\eps}\|_{\infty}\eps^3 \Exp{\sigma_{\eps}}\|\nabla^2 \psi_{\eps}\|_{HS}^{2}.
    \end{align*}
\end{proof}

With Lemmas \ref{lem:BP_upperbounds}-\ref{lem:LD_BP_covergence} established, we can now quickly prove Theorems \ref{thm:score-function-approx}- \ref{thm:one_step_convergence}. 

\begin{proof}[Proof of Theorem \ref{thm:one_step_convergence}]
This result follows from Lemmas \ref{lem:BP_upperbounds}, \ref{lem:bound-preserved-eps}, and \ref{lem:LD_BP_covergence}. Lemmas \ref{lem:BP_upperbounds} and \ref{lem:bound-preserved-eps} together prove an upper bound on $\Was{2}(\SB{\rho}{\eps}{},\LD{\rho}{\eps}{})$ that show, in particular,
$\lim\limits_{\eps \downarrow 0} \frac{1}{\eps}\Was{2}(\SB{\rho}{\eps}{},\LD{\rho}{\eps}{})=0,
$
and Lemma \ref{lem:LD_BP_covergence} proves an upper bound on $\Was{2}(\LD{\rho}{\eps}{},\opt{\eps}{1}{\rho})$ which implies 
$
\lim\limits_{\eps \downarrow 0} \frac{1}{\eps}\Was{2}(\LD{\rho}{\eps}{},\opt{\eps}{1}{\rho})=0.
$
That is, the three updates $\SB{\rho}{\eps}{}$, $\LD{\rho}{\eps}{}$, and $\opt{\eps}{1}{\rho}$ are within $o(\eps)$ of each other in the Wasserstein-2 metric. Theorem \ref{thm:one_step_convergence} now follows from the triangle inequality and simplifying the choice of the constant $K$.
\end{proof}

\begin{proof}[Proof of Theorem \ref{thm:score-function-approx}]
    Theorem \ref{thm:score-function-approx} follows directly from Theorem \ref{thm:one_step_convergence} by setting $\sigma_{\eps} = \rho$ for all $\eps > 0$. More precisely, the $\Was{2}$ bounds in Lemmas \ref{lem:BP_upperbounds} and \ref{lem:LD_BP_covergence} are obtained by bounding the $\mathbf{L}^2(\rho)$ convergence rates appearing in the statement of Theorem \ref{thm:score-function-approx}. The convergence of the gradient of the entropic potentials follows from \eqref{eq:bp-potential-identity}. 
\end{proof}

\begin{proof}[Proof of Theorem \ref{thm:sb-generator-approx}]
The proof follows from a slight modification of Lemma \ref{lem:BP_upperbounds}. Fix $\xi: \mathbb{R}^{d} \to \mathbb{R}$ as specified in the Theorem statement. Observe that
\begin{align*}
    \Exp{\SBstatic{\rho}{\eps}}[\xi(Y)|X=x]-\xi(x) &= \left(\Exp{\SBstatic{\rho}{\eps}}[\xi(Y)|X=x]-\Exp{\LDstatic{\rho}{\eps}}[\xi(Y)|X=x]\right)\\
    &+\left(\Exp{\LDstatic{\rho}{\eps}}[\xi(Y)|X=x]-\xi(x)\right).
\end{align*}
As $\xi \in \mathcal{D}(L)$, it holds in $\textbf{L}^{2}(\rho)$ that
    $\frac{1}{\eps}\left(\Exp{\LDstatic{\rho}{\eps}}[\xi(Y)|X=x]-\xi(x)\right) \to L \xi(x)$
as $\eps \downarrow 0$. It remains to show that $\frac{1}{\eps}\left(\Exp{\SBstatic{\rho}{\eps}}[\xi(Y)|X=x]-\Exp{\LDstatic{\rho}{\eps}}[\xi(Y)|X=x]\right)$ vanishes. Let $C > 0$ denote the Lipschitz constant of $\xi$, and let $(X,Y,Z)$ be the coupling constructed in Lemma \ref{lem:BP_upperbounds}. By the same argument in Lemma \ref{lem:BP_upperbounds} (i.e.\ Jensen and Talagrand), 
\begin{align*}
    \frac{1}{\eps}\abs{\Exp{\SBstatic{\rho}{\eps}}[\xi(Y)|X=x]-\Exp{\LDstatic{\rho}{\eps}}[\xi(Y)|X=x]} &= \frac{1}{\eps}\abs{\Exp{}[\xi(Y)|X=x]-\Exp{}[\xi(Z)|X=x]} \\
    &\leq \frac{C}{\eps}\sqrt{\alpha H(p(x,\cdot)|q_{\eps}(x,\cdot))}.
\end{align*}
Hence, by Jensen's inequality and Theorem \ref{thm:ld-sb-general-g},
\begin{align*}
    \limsup\limits_{\eps \downarrow 0}\frac{1}{\eps}\norm{\Exp{\SBstatic{\rho}{\eps}}[\xi(Y)|X=x]-\Exp{\LDstatic{\rho}{\eps}}[\xi(Y)|X=x]}_{L^{2}(\rho)} &\leq \lim\limits_{\eps \downarrow 0} \frac{C}{\eps}\sqrt{\alpha H(\SBstatic{\rho}{\eps}|\LDstatic{\rho}{\eps})} = 0.
\end{align*}  
\end{proof}

\bibliographystyle{alpha}
\bibliography{sample}

\end{document}